\documentclass[10pt]{article}
\usepackage{lipsum}
\usepackage{amsfonts}
\usepackage{algorithmic}
\usepackage{epsfig}
\usepackage{amsmath}
\usepackage{amsfonts}
\usepackage{amsthm}
\usepackage{amsbsy}
\usepackage{appendix}
\usepackage{graphicx}
\usepackage{epstopdf}
\usepackage[table]{xcolor}
\usepackage{subfloat}
\usepackage{float}
\usepackage[thinlines]{easytable}
\usepackage{array}
\usepackage{color}
\usepackage[caption=false]{subfig} 
\usepackage{algorithm}%
\usepackage{caption}
\usepackage{amsopn}
\usepackage{hyperref}
\usepackage[noabbrev,capitalise,nameinlink]{cleveref}
\hypersetup{colorlinks={true},linkcolor={blue},citecolor=blue}

\numberwithin{equation}{section}
\allowdisplaybreaks

\theoremstyle{plain}
\newtheorem{lemma}{Lemma}[section]
\newtheorem{theorem}{Theorem}[section]
\newtheorem{assumption}{\bf Assumption}[section]
\newtheorem{definition}{\bf Definition}[section]
\theoremstyle{remark}
\newtheorem{remark}{\bf Remark}[section]
\theoremstyle{remark}

\title {Predictive Accuracy of Dynamic Mode Decomposition\footnote{This work was supported in part by Defense Advanced Research Project Agency under award number 101513612, by Air Force Office of Scientific Research under award number FA9550-17-1-0417, and by U.S. Department of Energy under award number DE-SC0019130.}}
\author{Hannah Lu\footnote{Department of Energy Resources Engineering, Stanford, CA 94305. USA (hannahlu{@}stanford.edu)} \, and Daniel M. Tartakovsky\footnote{Department of Energy Resources Engineering, Stanford, CA 94305. USA (tartakovsky@stanford.edu)} }
\begin{document}

\maketitle

\begin{center}
{\bf Abstract}
\end{center}

Dynamic mode decomposition (DMD), which the family of singular-value decompositions (SVD), is a popular tool of data-driven regression. While multiple numerical tests demonstrated the power and efficiency of DMD in representing data (i.e., in the interpolation mode), applications of DMD as a predictive tool (i.e., in the extrapolation mode) are scarce. This is due, in part, to the lack of rigorous error estimators for DMD-based predictions. We provide a theoretical error estimator for DMD extrapolation of numerical solutions to linear and nonlinear parabolic equations. This error analysis allows one to monitor and control the errors associated with DMD-based temporal extrapolation of  numerical solutions to parabolic differential equations.  We use several computational experiments to verify the robustness of our error estimators and to compare the predictive ability of DMD with that of proper orthogonal decomposition (POD), another member of the SVD family. Our analysis demonstrates the importance of a proper selection of observables, as predicted by the Koopman operator theory. In all the tests considered, DMD outperformed POD in terms of efficiency due to its iteration-free feature. In some of these experiments, POD proved to be more accurate than DMD. This suggests that DMD is preferable for obtaining a fast prediction with slightly lower accuracy, while POD should be used if the accuracy is paramount. 
\\


\section{Introduction}
\label{sec:intro}
Dynamic mode decomposition (DMD)~\cite{kutzbook} has recently become a popular tool of data-driven regression. It belongs to the family of singular-value decompositions (SVD) and has its origins in representation of complex fluid flows in terms of their spatial modes and temporal frequencies~\cite{schmid2010dynamic}. This strategy for representation of spatiotemporal coherent structures has since been used for data diagnostics and related applications including video processing~\cite{kutz2016dynamic}, interpretation of neural activity measurements~\cite{brunton2016extracting}, financial trading~\cite{mann2016dynamic}, and forecast of infectious decease spreading~\cite{proctor2015discovering}. DMD with control has been developed to extract the input-output characteristics of dynamic systems with external control~\cite{proctor2016dynamic}. It has also been deployed to learn models of high-dimensional complex systems from data~\cite{li2017extended,rudy2017data,williams2015data}, in the spirit of equation-free simulations~\cite{kevrekidis2003equation}.

DMD is connected to interpretation of nonlinear dynamical systems via the Koopman operator theory~\cite{rowley2009spectral,mezic2013analysis}. The latter provides a bridge between finite-dimensional nonlinear dynamics and infinite-dimensional linear systems by observable functions ~\cite{koopman1931hamiltonian}. Theoretical studies of the DMD approximation to eigenvalues and eigenfunctions of the infinite-dimensional Koopman operator show that the performance of this finite eigen-approximation depends crucially on the choice of observable functions, requiring expert prior knowledge of the underlying dynamics~\cite{rowley2009spectral,williams2015data}. Machine learning techniques have been used to select the observable by identifying relevant terms in the dynamics from data~\cite{brunton2016discovering,schmidt2009distilling,wang2011predicting}. Extended DMD or EDMD employs regression from a dictionary of observables that spans a subspace of the space of scalar observables~\cite{williams2015data}.

Numerical implementations of DMD are also undergoing modifications and enhancements. Under various assumptions on the data, many variants of the standard DMD algorithm were introduced to compute the eigenvalues and DMD modes in more accurate and efficient ways~\cite{dawson2016characterizing, drmac2018data}. Sparsity-promoting DMD and compressed DMD combine DMD with sparsity techniques and modern theory of compressed sensing~\cite{brunton2015compressed,tu2014spectral}.  Inspired by the applications of DMD in video processing, multi-resolution DMD or mrDMD provides a means for recursive computation of DMD of separate spatiotemporal features at different scales in the data~\cite{kutz2016multiresolution}. The mrDMD approach preserves the translational and rotational invariances, which remains the Achilles heel of many SVD-based methods~\cite{rowley2000reconstruction}.

While multiple numerical tests demonstrated the power and efficiency of DMD in representing data (i.e., for interpolation), applications of DMD as a predictive tool (i.e., for extrapolation) are scarce. This is due, in part, to the lack of rigorous error estimators for DMD-based predictions. The convergence of DMD predictions are reported in~\cite{duke2012error} from the numerical perspective and in~\cite{korda2018convergence} from the theoretic perspective. A goal of our analysis is to provide a theoretical error estimator for DMD extrapolation of numerical solutions to linear and nonlinear parabolic equations. We are aware of no other quantitative analysis of the accuracy of DMD predictions. This error analysis allows one to monitor and control the errors associated with DMD-based temporal extrapolation of  numerical solutions to parabolic differential equations. That, in turn, would facilitate the design of efficient algorithms for multi-scale/multi-physics simulations.

An alternative way to predict future states of a system relies on reduced-order models (ROMs), which are constructed with the proper orthogonal decomposition (POD)~\cite{kerschen2005method, rowley2005model}. Time integration is still needed to compute future states, but only in a low-dimensional surrogate model. Thus, the computational cost is reduced and future states are predicted using the ROM derived from projecting the dynamics of the full system onto the hyperplane that the POD extracts from data. POD is an SVD-based method that is closely related to the principle component analysis (PCA) and the Karhuen-Lo\`eve transform. Recently, the empirical interpolation method (EIM)~\cite{barrault2004empirical} and the discrete empirical interpolation method (DEIM)~\cite{chaturantabut2010nonlinear} were combined with POD in order to overcome the difficulty of handling nonlinearities in ROM. Although the POD-EIM/DEIM methods lack error estimation, they have been used in various fields with satisfactory accuracy. While both POD and DMD are based on SVD, they provide two independent approaches to constructing ROMs. It is therefore worthwhile to compare their relative performance in terms of accuracy and efficiency. Advantages of hybridizing the two methods have been demonstrated in several numerical tests~\cite{alla2017nonlinear,williams2013hybrid}.

The paper is organized as follows: In~\cref{sec:dmd}, we formulate the DMD algorithm for the linear and nonlinear diffusion equations as a problem set up and provide a brief review of the DMD method and its connection to the Koopman operator theory. Our selection of the observables is also demonstrated with prior knowledge of underlying physics. In~\cref{sec:acc}, we present our main results in error estimation. Several numerical tests are presented in~\cref{sec:tests} to verify the error bound and the efficiency of DMD in prediction.  DMD and POD are compared in terms of their computational costs and accuracy. We summarize the results with a discussion of applications, challenges and future work in~\cref{sec:con}.

\section{Dynamic Mode Decomposition and Koopman Operators}
\label{sec:dmd}
Consider a state variable $u(\mathbf x,t) : \mathcal D \times \mathbb R^+ \rightarrow \mathbb R$ whose dynamics are governed by parabolic partial differential equation (PDE),
\begin{equation}\label{eq:2-1}
\partial_t u =\mathcal N(u)+f, \qquad \mathbf x \in \mathcal D \subset \mathbb R^d, \quad t > 0,
\end{equation}
where $\mathcal N$ is a linear or nonlinear differential operator representing the internal dynamics in $d$-dimensional space, and the linear or nonlinear source term  $f$ represents the external source/sink into the system. 
Discretization of the simulation domain $\mathcal D$ into $N$ elements or nodes ($N\gg 1$) transforms the PDE~\cref{eq:2-1} into either a high-dimensional linear dynamical system
\begin{equation}\label{eq:2-2}
\frac{\text d \mathbf u}{\text dt} = \boldsymbol{\mathcal A } \mathbf u+\mathbf f,
\end{equation}
or a high-dimensional nonlinear dynamical system
\begin{equation}\label{eq:2-3}
\frac{\text d \mathbf u}{\text dt} = \boldsymbol \Psi (\bold u)+\bold f,
\end{equation}
where $\mathbf u = [u(\mathbf x_1,t),\cdots,u(\mathbf x_N,t)]^\top$ is the spatial discretization of $u(\mathbf x,t)$; $\boldsymbol{\mathcal A}$ and $\boldsymbol \Psi$ are linear and nonlinear operators on $\mathbb R^N$, respectively; and $\mathbf f$ represents the correspondingly discretized source term $f$.

Low-dimensional ROMs are often used to reduce the computational cost of solving the high-dimensional systems~\cref{eq:2-2,eq:2-3}. For example, POD has been deployed to construct accurate and efficient ROMs for~\cref{eq:2-2}~\cite{kerschen2005method,rowley2005model}. Time evolution of $u(\mathbf x,t)$ needs to be computed but only in a small subspace of the original high-dimensional space. For nonlinear systems~\cref{eq:2-3}, construction of a right ROM using POD becomes more challenging and requires some modifications, such as empirical interpolation method (EIM)~\cite{barrault2004empirical} and discrete empirical interpolation method (DEIM)~\cite{chaturantabut2010nonlinear}, whose accuracy cannot be determined a priori. To the best of our knowledge, error estimates of POD-EIM/DEIM are lacking unless the fully resolved solution is available.

The DMD method aims to approximate the eigenvalues and eigenfunctions of $\mathcal A$ in~\cref{eq:2-2} and provides an alternative to POD in solving large linear systems. A major advantage of DMD over POD is its equation-free nature, which allows future-state predictions without any computation of further time evolution. For the nonlinear problems~\cref{eq:2-3}, DMD seeks a finite-dimensional approximation of the infinite-dimensional Koopman operator of the nonlinear dynamics. With carefully chosen observables, a ROM can be constructed in the observable space with sufficient accuracy. We briefly review DMD and the related Koopman operator theory in~\cref{sec:dmd_dmd,sec:dmd_koopman} as a set up for the accuracy analysis in~\cref{sec:acc}.

\subsection{Dynamic mode decomposition}
\label{sec:dmd_dmd}

Temporal discretization of~\cref{eq:2-2} with time step $\Delta t$ yields
\begin{equation}\label{eq:2-4}
\mathbf u^{n+1} = \boldsymbol{\mathcal A}\mathbf  u^n+\Delta t\bold f^*, \qquad n \ge 0,
\end{equation}
where $\boldsymbol{\mathcal A}$ is an $N\times N$ matrix and $\bold f^*$ is, e.g., interpolation of $\bold f^n$ and $\bold f^{n+1}$. This is rewritten as
\begin{equation}\label{eq:2-5}
\bold u^{n+1}  = \bold K \bold u^n, \qquad n \ge 0,
\end{equation}
where $\bold K$ is an $N$-dimensional linear operator. The fully resolved model~\cref{eq:2-4} is advanced by $m$ time steps  and the resulting temporal snapshots of $\mathbf u(t)$ is recorded in two matrices:
\begin{equation}\label{eq:2-6}
\bold X = \begin{bmatrix}
|&|&&|\\
\bold u^0&\bold u^1&\cdots&\bold u^{m-1}\\
|&|&&|
\end{bmatrix} \quad\text{and}\quad \bold X' = \begin{bmatrix}
|&|&&|\\
\bold u^1&\bold u^2&\cdots&\bold u^{m}\\
|&|&&|
\end{bmatrix}.
\end{equation}
Using these two data sets, one approximates the eigenvalues and eigenvectors of $\bold K$ using~\cref{alg:dmd_state}.

\begin{algorithm}
\caption{DMD algorithm on state space~\cite{kutzbook}}
\label{alg:dmd_state}
\begin{itemize}
\item[1.] Apply Singular Value Decomposition (SVD) $\bold X \approx \bold U\boldsymbol\Sigma \bold V^*$, where $\bold U \in \mathbb C^{N\times r}$ is a unitary matrix, $\boldsymbol \Sigma \in \mathbb C^{r\times r}$ is a diagonal matrix with components $\sigma_k \ge 0$ that are called singular values of $\mathbf X$, $\mathbf V^*$ is the conjugate transpose of unitary matrix $\bold V\in \mathbb C^{r\times m}$, and $r$ is the truncated rank chosen by certain criteria.
\item [2.] Compute $\tilde{\bold K}=\bold U^*\bold X'\bold V\boldsymbol\Sigma^{-1}$ ; use it as a low-rank ($r\times r$) approximation of $\bold K$.
\item [3.] Compute eigendecomposition of $\tilde{\bold K}$: $\tilde{\bold K} \bold W = \bold W\boldsymbol\Lambda$, where $\boldsymbol\Lambda = (\lambda_k)$ are eigenvalues and columns of $\mathbf W$ are the corresponding eigenvectors.
\item [4.] Eigenvalues of $\bold K$ can be approximated by $\boldsymbol\Lambda$ with corresponding eigenvectors in the columns of $\boldsymbol\Phi  = \bold X'\bold V\boldsymbol\Sigma^{-1}\bold W$.
\end{itemize}
\end{algorithm}

Each column of $\boldsymbol\Phi$ in~\cref{alg:dmd_state} is a DMD mode corresponding to a particular eigenvalue in $\boldsymbol\Lambda$. With the approximated eigenvalues and eigenvectors of $\bold K$ in hand, a solution at the $(n+1)$-th time step ($n>m$) is constructed analytically as
\begin{equation}\label{eq:2-7}
\bold u_\text{DMD}^{n+1} =\boldsymbol\Phi\boldsymbol\Lambda^{n+1}\bold b,  \qquad n>m,
\end{equation}
where $\bold b =\boldsymbol\Phi^{-1}\bold u^{0}$ is an $r\times1$ vector representing the initial amplitude of each mode. Notice that no more iteration is needed in the prediction. The solution at any future time is approximated directly with~\cref{eq:2-7} using only information encapsulated in the first $m$ temporal snapshots.

\subsection{Koopman Operator Theory}
\label{sec:dmd_koopman}
The nonlinear dynamical system~\cref{eq:2-3} belongs to a general class  of dynamical systems,
\begin{equation}\label{eq:2-8}
\frac{\text d\bold u}{\text dt} = \mathcal N(\bold u),
\end{equation}
where the state $\bold u \in \mathcal M \subset \mathbb R^N$ is defined on a smooth $N$-dimensional manifold $\mathcal M$, and $\mathcal N$ is a finite-dimensional nonlinear operator. Given a flow map $\mathcal N_t :\mathcal M \to \mathcal M$,
\begin{equation}\label{eq:2-9}
\mathcal N_t (\bold u(t_0)) = \bold u(t_0+t) = \bold u(t_0)+\int_{t_0}^{t_0+t} \mathcal N(\bold u(\tau)) \text d \tau,
\end{equation}
the corresponding discrete-time dynamical system is described by
\begin{equation}\label{eq:2-10}
\bold u^{n+1} = \mathcal N_t(\bold u^n).
\end{equation}

\begin{definition}[Koopman operator~\cite{kutzbook}]
For nonlinear dynamic system~\cref{eq:2-8}, the Koopman operator $\mathcal K$ is an infinite-dimensional linear operator that acts on all observable functions $g: \mathcal M\to \mathbb C$ so that
\begin{equation}\label{eq:2-11}
\mathcal K g(\bold u) = g(\mathcal N(\bold u)).
\end{equation}
For discrete dynamic system~\cref{eq:2-10}, the discrete-time Koopman operator $\mathcal K_t$ is 
\begin{equation}\label{eq:2-12}
\mathcal K_t g(\bold u^{n}) = g(\mathcal N_t(\bold u^n)) = g(\bold u^{n+1}).
\end{equation}
\end{definition}
The Koopman operator transforms the finite-dimensional nonlinear problem~\cref{eq:2-10} in the state space into the infinite-dimensional linear problem~\cref{eq:2-12} in the observable space. Since $\mathcal K_t$ is an infinite-dimensional linear operator, it is equipped with infinite eigenvalues $\{\lambda_k\}_{k=1}^{\infty}$ and eigenfunctions $\{\phi_k\}_{k=1}^\infty$. In practice, one has to make a finite approximation of the eigenvalues and eigenfunctions. The following assumption is essential to both a finite-dimensional approximation and the choice of observables.

\begin{assumption}\label{ass:a2}
Let $\mathbf y$ denote a $p \times 1$ vector of observables,
\begin{equation}\label{eq:2-13}
\bold y^n = \bold g(\bold u^{n}) = \begin{bmatrix}
g_1(\bold u^n)\\
\vdots\\
g_p(\bold u^n)
\end{bmatrix},
\end{equation} 
where $g_j: \mathcal M \to \mathbb C$ is an observable function, with $j =1,\cdots, p$. If the chosen observable $\bold g$ is restricted to an invariant subspace spanned by eigenfunctions of the Koopman operator $\mathcal K_t$, then it induces a linear operator $\bold K$ that is finite-dimensional and advances these eigenobservable functions on this subspace~\cite{brunton2016koopman}.
\end{assumption}

Based on \cref{ass:a2}, the DMD algorithm can be deployed to approximate the eigenvalues and eigenfunctions of $\bold K$ using the collected temporal snapshots in the observable space. This DMD strategy is implemented in \cref{algorithm:phy_dmd}.

\begin{algorithm}
\caption{DMD algorithm on observable space~\cite{kutzbook}}
\begin{itemize}
\item[0.] Create the data matrices of observables
\begin{equation}\label{eq:2-14}
\bold Y = \begin{bmatrix}
|&|&&|\\
\bold y^0&\bold y^1&\cdots&\bold y^{m-1}\\
|&|&&|
\end{bmatrix} \quad\text{and}\quad
\bold Y' = \begin{bmatrix}
|&|&&|\\
\bold y^1&\bold y^2&\cdots&\bold y^{m}\\
|&|&&|
\end{bmatrix}
\end{equation}
where each column is given by $\bold y^k = \bold g(\bold u^k)$.
\item[1.] Apply SVD $\bold Y \approx \bold U\boldsymbol\Sigma \bold V^*$ with $\bold U \in \mathbb C^{p\times r}, \boldsymbol\Sigma \in \mathbb C^{r\times r}, \bold V\in \mathbb C^{r\times m}$, where $r$ is the truncated rank chosen by certain criteria.
\item [2.] Compute $\tilde{\bold K}=\bold U^*\bold X'\bold V\boldsymbol\Sigma^{-1}$ as a $r\times r$ low-rank approximation for $\bold K$.
\item [3.] Compute eigendecomposition of $\tilde{\bold K}$: $\tilde{\bold K} \bold W = \bold W\boldsymbol\Lambda$, $\boldsymbol\Lambda = (\lambda_k)$.
\item [4.] Reconstruct eigendecomposition of $\bold K$. Eigenvalues are $\boldsymbol \Lambda$ and eigenvectors are $\boldsymbol\Phi  = \bold X'\bold V\boldsymbol\Sigma^{-1}\bold W$.
\item [5.] Predict future $\bold y_\text{DMD}^{n+1}$ as
\begin{equation}\label{eq:2-15}
\bold y_\text{DMD}^{n+1} = \boldsymbol\Phi\Lambda^{n+1} \bold b, \quad \bold b = \Phi^{-1}\bold y^0 \qquad\mbox{for}\  n>m.
\end{equation}
\item [6.] Transform from observables space back to the state space,
\begin{equation}\label{eq:2-16}
\bold u_\text{DMD}^n =\bold g^{-1}(\bold y_\text{DMD}^n).
\end{equation}
\end{itemize}
\label{algorithm:phy_dmd}
\end{algorithm}

\begin{remark}
Connections between the DMD theory and the Koopman spectral analysis under specific conditions on the observables and collected data are established by a theorem in~\cite{tu2013dynamic}. This theorem indicates that judicious selection of the observables  is critical to success of the Koopman method.
\end{remark}

\begin{remark}
In general, there is no principled way to select observables without expert knowledge of a dynamical system. Machine learning techniques can be deployed to identify relevant terms in the dynamics from data, which guide selection of the observables~\cite{schmidt2009distilling,wang2011predicting,brunton2016discovering}.
\end{remark}

\section{Analysis of Predictive Accuracy}
\label{sec:acc}
We use a resolved accurate solution of~\cref{eq:2-5} under a certain CFL condition as a reference or yardstick against which to test the accuracy of the DMD prediction~\cref{eq:2-7}. 

\subsection{Preliminaries}
\label{sec:acc_pre}

Here we provide a brief summary of the key results relevant to our subsequent analysis.

\begin{assumption}\label{ass:a3-1}
Let $\{\lambda_1,\lambda_2,\cdots,\lambda_N\}$ be the eigenvalues of $\boldsymbol{\mathcal A}$ in~\cref{eq:2-4}. We assume 
\begin{equation}\label{eq:3-2}
\max_{1\leq k\leq N}|\lambda_k|\leq 1.
\end{equation}
\end{assumption}

\begin{lemma}
\label{lemma:l2}
Under~\cref{ass:a3-1}, any stable numerical method of~\cref{eq:2-4} satisfies the maximum principle in the discrete setting, i.e.,
\begin{equation}\label{eq:3-3}
\begin{aligned}
\|\bold u^{n+1}\|_2&\leq \|\bold u^n\|_2+\Delta t\max\{\|\bold f^n\|_2,\|\bold f^{n+1}\|_2\}\\
&\leq \cdots\\
& \leq \|\bold u^0\|_2+\Delta t\max\{\sum_{k=0}^n \|\bold f^k\|_2,\sum_{k=1}^{n+1} \|\bold f^k\|_2\}.
\end{aligned}
\end{equation}
\end{lemma}
\begin{proof}
\begin{equation}\label{eq:3-4}
\begin{aligned}
\|\bold u^{n+1}\|_2&\leq \|\bold A\bold u^n\|_2+\Delta t\|\bold f^*\|_2\\
&\leq \|\bold A\|_2\|\bold u^n\|_2+\Delta t\|\bold f^n\|_2\\
&=\rho(\bold A)\|\bold u^n\|_2+\Delta t\max\{\|\bold f^n\|_2,\|\bold f^{n+1}\|_2\}\\
&=\max_{1\leq k\leq N}|\lambda_k(\bold A)|\|\bold u^n\|_2+\Delta t\max\{\|\bold f^n\|_2,\|\bold f^{n+1}\|_2\}.
\end{aligned}
\end{equation}
According to~\cref{ass:a3-1},
\begin{equation}\label{eq:3-5}
\max_{1\leq k\leq N}|\lambda_k(\bold A)|<1.
\end{equation}
Thus,
\begin{equation}\label{eq:3-6}
\|\bold u^{n+1}\|_2\leq \|\bold u^n\|_2+\Delta t\max\{\|\bold f^n\|_2,\|\bold f^{n+1}\|_2\},
\end{equation}
and~\cref{lemma:l2} holds.
\end{proof}

\begin{lemma}
\label{lemma:l3}
DMD on $m$ temporal snapshots is designed such that $\|\bold u^{m}-\bold u_\text{DMD}^{m}\|_2$ is minimized.
\end{lemma}
\begin{proof}
See~\cite{schmid2010dynamic,drmac2018data}.
\end{proof}

\subsection{Main Results}
We rewrite the DMD prediction~\cref{eq:2-7} as
\begin{equation}\label{eq:3-7}
\begin{aligned}
\bold u_\text{DMD}^{n+1}& = \boldsymbol\Phi\boldsymbol\Lambda^{n+1}\bold b \\
&=\boldsymbol\Phi\boldsymbol\Lambda\boldsymbol\Phi^{-1}\boldsymbol\Phi\boldsymbol\Lambda^{n}\bold b \\
&= \boldsymbol\Phi\boldsymbol\Lambda\boldsymbol\Phi^{-1}\bold u_\text{DMD}^n\\
&=\bold u_\text{DMD}^n +(\boldsymbol\Phi\boldsymbol\Lambda\boldsymbol\Phi^{-1}-\bold I_{N\times N})\bold u_\text{DMD}^n \\
&=\bold u_\text{DMD}^n +\bold B\bold u_\text{DMD}^n.
\end{aligned}
\end{equation}
Here $\bold B =\boldsymbol\Phi\boldsymbol\Lambda\boldsymbol\Phi^{-1}-\bold I_{N\times N}$, where $\boldsymbol\Phi$ is an $N\times r$ matrix and $\boldsymbol\Phi^{-1}$ is an $r\times N$ matrix $\boldsymbol\Phi^{-1}$ defined as $\boldsymbol\Phi^{-1}\boldsymbol\Phi =\bold I_{r\times r}$.

\begin{theorem}\label{thm:t3-4}
Define the local truncation error
\begin{equation}\label{eq:3-8}
\boldsymbol \tau^n = \bold u^n-\bold u^{n-1}-\bold B\bold u^{n-1}.
\end{equation}
Then, for any $n\geq m$,
\begin{equation}\label{eq:3-9}
\|\boldsymbol \tau^n\|_2\leq \varepsilon_m,
\end{equation}
where the constant $\varepsilon_m$ depends only on the number of snapshots $m$.
\end{theorem}
\begin{proof}
\begin{equation}\label{eq:3-10}
\begin{aligned}
\|\boldsymbol \tau^n\|_2 =&\| \bold u^n-\bold u^{n-1}-\bold B\bold u^{n-1}\|_2\\
=&\| \bold u^n-\bold u^{n-1}-(\boldsymbol\Phi\boldsymbol\Lambda\boldsymbol\Phi^{-1}-\bold I_{N\times N})\bold u^{n-1}\|_2\\
=&\| \bold u^n-\boldsymbol\Phi\boldsymbol\Lambda\boldsymbol\Phi^{-1}\bold u^{n-1}\|_2\\
=&\|(\bold K-\boldsymbol\Phi\boldsymbol\Lambda\boldsymbol\Phi^{-1})\bold u^{n-1}\|_2\\
\leq&\|\bold K-\boldsymbol\Phi\boldsymbol\Lambda\boldsymbol\Phi^{-1}\|_F\|\bold u^{n-1}\|_2.
\end{aligned}
\end{equation}
Since $\|\bold K-\boldsymbol\Phi\boldsymbol\Lambda\boldsymbol\Phi^{-1}\|_F\leq c_m$ where $c_m$ is a constant depending on the number of snapshots $m$,~\cref{thm:t3-4} holds with
\begin{equation}\label{eq:3-11}
\varepsilon_m = c_m(\|\bold u^0\|_2+\Delta t\max\{\sum_{k=0}^{n-1}\|\bold f^k\|_2,\sum_{k=1}^{n}\|\bold f^k\|_2\}).
\end{equation}
\end{proof}

\begin{remark}
The value of $c_m$ decreases to $0$ as $m$ increases and so does $\varepsilon_m$. In the limit of large number of snapshots, $\boldsymbol\Lambda$ and $\boldsymbol\Phi$ become the exact eigenvalues and eigenvectors of $\bold K$. Then 
\begin{equation}\label{eq:3-12}
\begin{aligned}
\|\bold K-\boldsymbol\Phi\boldsymbol\Lambda\boldsymbol\Phi^{-1}\|_F &= \sup_{z\in \mathbb R^{N}\setminus\{ 0\}}\frac{\|\bold Kz-\boldsymbol\Phi\boldsymbol\Lambda \boldsymbol\Phi^{-1}z\|_2}{\|z\|_2}\\
&= \sup_{w\in \mathbb R^{N}\setminus\{ 0\}}\frac{\|\bold K \boldsymbol\Phi w-\boldsymbol\Phi\boldsymbol\Lambda w\|_2}{\|\boldsymbol\Phi w\|_2}\\
&= \sup_{w\in \mathbb R^{N}\setminus\{ 0\}}\frac{\|\boldsymbol\Lambda \boldsymbol\Phi w-\boldsymbol\Phi\boldsymbol\Lambda w\|_2}{\|\boldsymbol\Phi w\|_2}\\
&=0.
\end{aligned}
\end{equation}
In other words, the more snapshots are obtained, the more accurate the approximation of $\bold K$ becomes. Thus, the local truncation error caused by replacing $\bold K$ with $\bold B$ can be minimized. A convergence proof of the eigenvalue and eigenfunction approximation of $\bold K$ by DMD and convergence from $\bold K \to \mathcal K_t$ can be found in \cite{korda2018convergence}.
\end{remark}

\begin{remark}
For fixed $m$,  the local trucncation error can be improved by refining the Ritz pairs in the DMD algorithm [DDMD-RRR]. See \cite{drmac2018data}.
\end{remark}

\begin{theorem}\label{thm:t3-7}
Define the global truncation error
\begin{equation}\label{eq:3-13}
\bold e^n = \bold u^n-\bold u_\mathrm{DMD}^n.
\end{equation}
Then, for $n\geq m$,
\begin{equation}\label{eq:3-14}
\|\bold e^n\|_2 \leq \|\boldsymbol\Phi^{-1}\|_F[\|\bold e^m\|_2+(n-m)\varepsilon_m].
\end{equation}
\end{theorem}

\begin{proof}
\begin{equation}\label{eq:3-15}
\begin{aligned}
\bold e^n =& \bold u^n-\bold u_\text{DMD}^n\\
=&\bold u^n-(\bold u_\text{DMD}^m+\bold B\bold u_\text{DMD}^m+\bold B\bold u_\text{DMD}^{m+1}+\cdots+\bold B\bold u_\text{DMD}^{n-1})\\
=&\bold u^n-\bold u^{n-1}+\bold u^{n-1}-(\bold u_\text{DMD}^m+\bold B\bold u_\text{DMD}^m+\bold B\bold u_\text{DMD}^{m+1}+\cdots+\bold B\bold u_\text{DMD}^{n-2}) 
-\bold B\bold u_\text{DMD}^{n-1}\\
=&\bold u^n-\bold u^{n-1}+\bold e^{n-1}-\bold B\bold u_\text{DMD}^{n-1}\\
=&\bold e^{n-1}+\bold u^n-\bold u^{n-1}-\bold B\bold u^{n-1}+\bold B\bold u^{n-1}-\bold B\bold u_\text{DMD}^{n-1}\\
=&\bold e^{n-1}+\boldsymbol \tau^{n}+\bold B\bold e^{n-1}\\
=&\boldsymbol \tau^n+\boldsymbol\Phi\boldsymbol\Lambda\boldsymbol\Phi^{-1}\bold e^{n-1}\\
=&\boldsymbol \tau^n+\boldsymbol\Phi\boldsymbol\Lambda\boldsymbol\Phi^{-1} (\boldsymbol \tau^{n-1}+\boldsymbol\Phi\boldsymbol\Lambda\boldsymbol\Phi^{-1}\bold e^{n-2})\\
=&\boldsymbol \tau^n+\boldsymbol\Phi\boldsymbol\Lambda\boldsymbol\Phi^{-1}\boldsymbol \tau^{n-1}+\boldsymbol\Phi\boldsymbol\Lambda^2\boldsymbol\Phi^{-1}\bold e^{n-2}\\
=&\cdots\\
=&\boldsymbol\Phi\boldsymbol\Lambda^{n-m}\boldsymbol\Phi^{-1}\bold e^m+\sum_{k=0}^{n-m-1}\boldsymbol\Phi\boldsymbol\Lambda^k\boldsymbol\Phi^{-1}\boldsymbol \tau^{n-k}.
\end{aligned}
\end{equation}
Then 
\begin{equation}\label{eq:3-16}
\begin{aligned}
\|\bold e^n\|_2&\leq \|\boldsymbol\Phi\boldsymbol\Lambda^{n-m}\boldsymbol\Phi^{-1}\|_F\|\bold e^m\|_2+(n-m)\varepsilon_m\max_{0\leq k\leq n-m-1}\|\boldsymbol\Phi\boldsymbol\Lambda^k\boldsymbol\Phi^{-1}\|_F\\
&\leq  \|\boldsymbol\Phi\boldsymbol\Lambda^{n-m}\|_F\|\boldsymbol\Phi^{-1}\|_F\|\bold e^m\|_2+(n-m)\varepsilon_m\max_{0\leq k\leq n-m-1}\|\boldsymbol\Phi\boldsymbol\Lambda^k\|_F\|\boldsymbol\Phi^{-1}\|_F.
\end{aligned}
\end{equation}
According to~\cref{lemma:l3}, $\|\bold e^m\|$ is fixed and minimal. Hence, if accuracy of the local truncation error is of $\mathcal O((\Delta t)^q)$, then the global truncation error is of $\mathcal O((\Delta t)^{q-1})$.
\end{proof}

\cref{thm:t3-7} provides quantitative error bounds of the DMD method with explicit error dependence. In complex simulations, one would not expect the DMD prediction from a local data set to capture the global dynamics accurately. Instead, one can  use the error bounds to set up a threshold for DMD prediction limits and combine a resolved algorithm with fast DMD prediction. This would considerably speed up the simulations.

\subsection{Application to nonlinear parabolic problems}

Consider a general nonlinear reaction-diffusion equation in $d$ spatial dimensions
\begin{equation}\label{eq:4-2}
\left\{
\begin{aligned}
&\partial_t u = \nabla \cdot [k\psi(u)\nabla u] + f(u), \qquad \mathbf x \in \mathcal D \subset \mathbb R^d, \quad t > 0 \\
&u(\mathbf x,0) = u_0(\mathbf x), \qquad \mathbf x \in \mathcal D,
\end{aligned}
\right.
\end{equation}
with non-negative functions $k = k(\mathbf x)$ and $\psi = \psi(u)$ whose product is diffusion coefficient $D(\mathbf x, u) = k(\mathbf x) \psi(u)$. Spatial discretization of~\cref{eq:4-2} leads to the corresponding high-dimensional nonlinear ODE~\cref{eq:2-3}. Its DMD treatment relies on the one's ability to identify informative observables and requires the prior knowledge of the structure of governing equations such as~\cref{eq:4-2}.  Examples in~\cref{sec:t2,sec:t3,sec:t4} illustrate the critical role of observable selection in the DMD method.

For~\cref{eq:4-2}, expert knowledge suggests the existence of a function $\eta(u)$ such that $\eta'(u) = \psi(u)$, which can be constructed via the Kirchhoff transform (e.g.,~\cite{tartakovsky-2003-stochastic, tartakovsky1999conditional}). Then, by chain rule, \cref{eq:4-2} is rewritten as
\begin{equation}\label{eq:4-3}
\left\{
\begin{aligned}
&\partial_t u -\nabla \cdot [k\nabla \eta(u) ] = f(u)\\
&u(x,0) = u_0(x),
\end{aligned}
\right.
\end{equation}
so that the nonlinear diffusion in $u$ becomes linear in $\eta$. Spatial discretization of~\cref{eq:4-3} leads to
\begin{equation}\label{eq:4-4}
\frac{\text d\bold u}{\text dt} = \boldsymbol{\mathcal A} \boldsymbol \eta(\bold u) +\bold F(\bold u)
\end{equation}
where $\boldsymbol{\mathcal A}$ is the same linear operator in~\cref{eq:2-2}. Motivated by the nonlinear observable choice for the nonlinear Schr\"odinger equation in~\cite{kutzbook}, and by the accurate and robust performance of DMD on linear diffusion reported below, we choose the observable 
\begin{equation}\label{eq:4-5}
\begin{aligned}
\bold g = [g_1(\bold u), \cdots, g_p(\bold u)],  \qquad \mbox{s.t.}  \ \ \bold u, \boldsymbol \eta(\bold u), \bold F(\bold u) \in \text{span}\{g_1(\bold u), \cdots, g_p(\bold u)\}.
\end{aligned}
\end{equation}

The reference solution of~\cref{eq:4-2} is obtained by discretizing~\cref{eq:4-4} in time,
\begin{equation}\label{eq:4-6}
\bold u^{n+1} =\bold u^n+\Delta t\boldsymbol \eta^*+\Delta t\bold F^*,
\end{equation}
where the superscript $*$ denotes linear interpolation between time $t^{n+1}$ and $t^n$. For the observables in~\cref{eq:4-5}, we have
\begin{equation}\label{eq:4-7}
\begin{aligned}
&\bold u^{n+1}, \boldsymbol \eta^{n+1}, \bold F^{n+1} \in \text{span} \{g_1(\bold u^{n+1}), \cdots, g_p(\bold u^{n+1})\},\\
&\bold u^{n}, \boldsymbol \eta^{n}, \bold F^{n} \in \text{span} \{g_1(\bold u^{n}), \cdots, g_p(\bold u^{n})\}.
\end{aligned}
\end{equation}
Thus, \cref{alg:dmd_state} induces a linear operator denoted by $\bold K$ such that
\begin{equation}\label{eq:4-8}
\bold y^{n+1} =\bold K \bold y^n,
\end{equation}
where $\bold y^n =\bold g(\bold u^n)$ defined in~\cref{eq:2-13}. Treating~\cref{eq:4-8} as reference solution, against which we compare the DMD prediction~\cref{eq:2-16}, one gets exactly the same formulae as~\cref{eq:2-5} and~\cref{eq:2-7} but in observable space:
\begin{equation}\label{eq:4-9}
\begin{aligned}
\bold y^{n+1} = & \; \bold K\bold y^n,\\
\bold y_\text{DMD}^{n+1} = & \; \Phi \Lambda^{n+1}\bold b.
\end{aligned}
\end{equation}
So the error analysis in~\cref{sec:acc} carries on in terms of $\bold y$.

\section{Numerical Tests}
\label{sec:tests}

We test the robustness of our error estimates and the DMD performance in the extrapolation regime on several test problems arranged in order of difficulty.

In our resolved simulations, we use finite difference in space and forward Euler in time with CFL condition $\Delta t\sim \mathcal O((\Delta x)^2)$. Although there are many relatively efficient implicit/semi-implicit solvers, the computational difficulty of solving high-dimensional systems iteratively remains essentially the same. We would regard them the same order of computational time and simply take the fully explicit discretization as the resolved solutions. In the following tests, $N=500$ spatial mesh is created in $x$ and $n=500$ solutions are uniformly selected from a specified time interval. Thus, the reference solution is built on this $500\times 500$ mesh. We also compare the relative performance of DMD and POD(-DEIM) in terms of both their computational time and error with respect to the reference solution.

\subsection{Linear diffusion}

We start with a linear diffusion equation,
\begin{subequations}\label{4-10}
\begin{equation}
\frac{\partial u}{\partial t} = \frac{\partial^2 u}{\partial x^2},  \qquad x \in[0,1], \qquad t\in [0,T]\\
\end{equation}
subject to several sets of initial and boundary conditions
\begin{align}
u(x,0) = u_0, \qquad u(0,t) = u_\text{L}, \qquad u(1,t) = 1.
\end{align}
\end{subequations}
Discretization of the spatial domain $[0,1]$ with a fine mesh of size $\Delta x \ll1$ gives rise to the equivalent high-dimensional ODE~\cref{eq:2-2}, where $\bold u = [u(x_1,t),\cdots,u(x_N,t)]^\top$ is the spatial discretization of $u(x,t)$ with $N\gg1$ and $\boldsymbol{\mathcal A}$ is a linear operator representing the diffusion. In this setting, \cref{ass:a3-1} certainly holds by the proof in Chapter~7 of~\cite{thomee1984galerkin}.

\subsubsection{Relaxation to equilibrium (Test~1a)}
Consider~\eqref{4-10} with $T = 0.2$, $u_0 = 0$, and $u_\text{L} = 0$.
%
\cref{fig:f1} demonstrates visual agreement between the true solution $u(x,t)$ and its counterpart predicted by DMD with $m=200$ temporal snapshots, the two solutions converge to the same stationary state.

\begin{figure}[tbhp]
 \centering
 \includegraphics{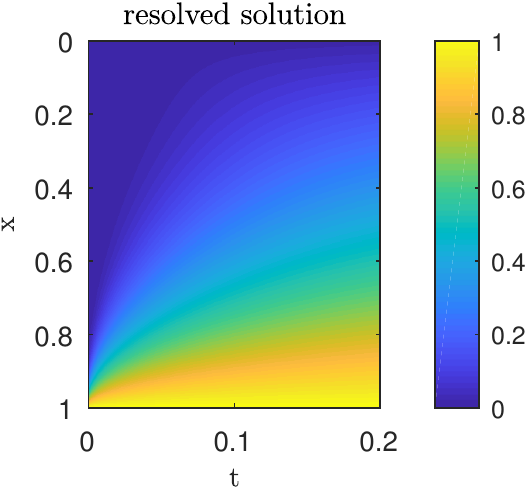}
\includegraphics{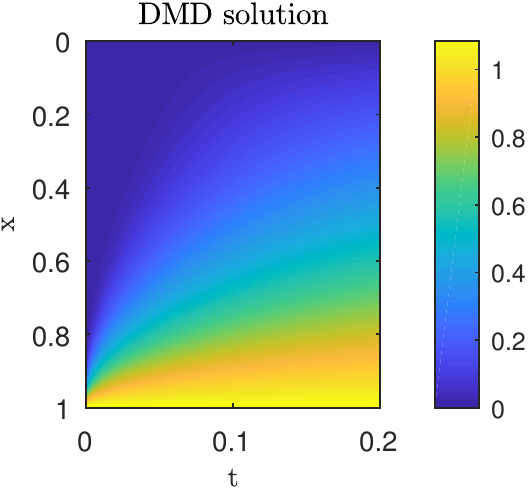}
\caption{Test~1a. Reference solution (left) and its DMD approximation with $m=200$ snapshots (right).}
\label{fig:f1}
\end{figure}

\Cref{fig:f2} exhibits the local truncation error $\tau$~\cref{eq:3-8} and the global truncation error $\bold e$~\cref{eq:3-13} of the DMD with $m=100$, $200$ and $300$ snapshots of the reference solution. The rank in step 1 of~\cref{alg:dmd_state}. is truncated by the criteria of 
\begin{equation}\label{eq:4-11}
r=\max\{i: \sigma_i> \epsilon \sigma_1\},
\end{equation}
where $\sigma_i$ are the diagonal elements of $\boldsymbol\Sigma$ in SVD. The figure shows that the local truncation errors decreases with the number of snapshots, resulting in a more accurate prediction. This is consistent with the intuition that DMD can better capture the dynamics by learning from richer/larger data sets.

\begin{figure}[tbhp]
\centering
\includegraphics{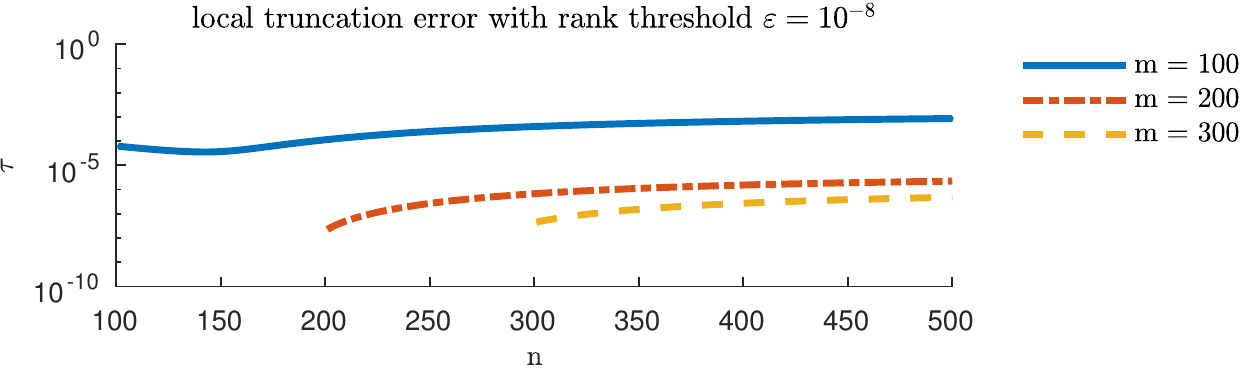}
\includegraphics{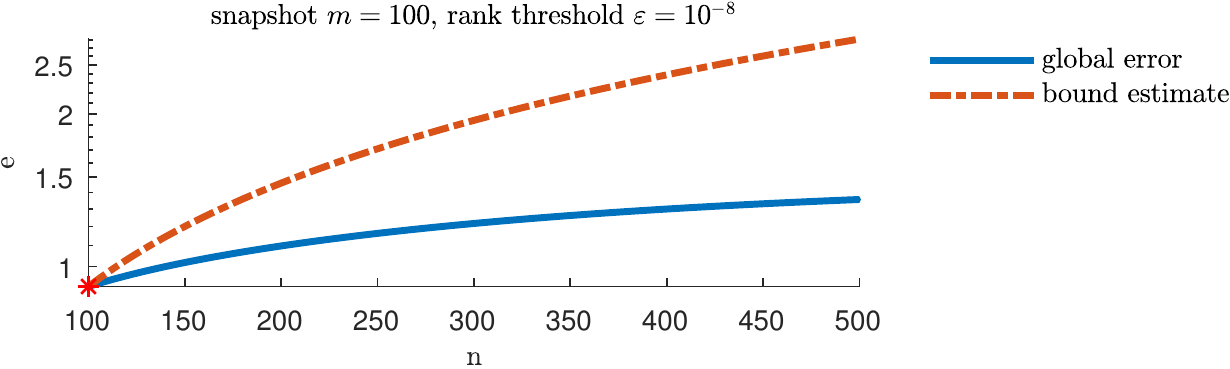}
\includegraphics{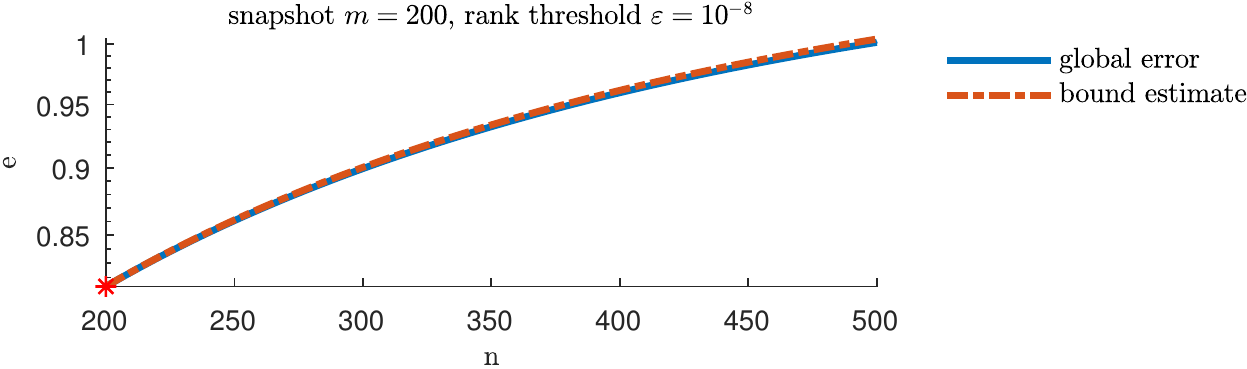}
\caption{Test~1a. Local truncation error $\tau$ for DMD with $m=100$, $200$ and $300$ snapshots (top); and global error $\bold e$ (error of the solution $u$) for DMD with $m=100$ (middle) and $m =200$ snapshots (bottom). The global error is negligible for $m = 300$ (not shown). The rank threshold is set to $\varepsilon = 10^{-8}$.}
\label{fig:f2}
\end{figure}

If a more stringent condition on the rank truncation is imposed, i.e., a relatively higher-order surrogate model is established, further reduction in both local and global errors is observed (\cref{fig:f4}). The good performance of DMD in Test 1a is not surprising: the monotonic (exponential) decay of the solution to the linear diffusion equation is captured by a relatively few temporal snapshots. The next example provides a more challenging test by introducing temporal fluctuations at the boundary $x=0$.

\begin{figure}[tbhp]
\centering
\includegraphics{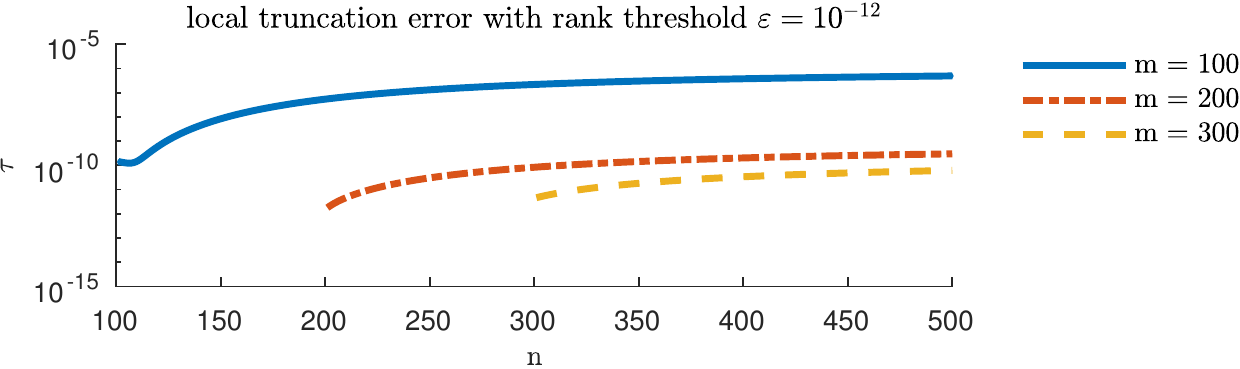}
\includegraphics{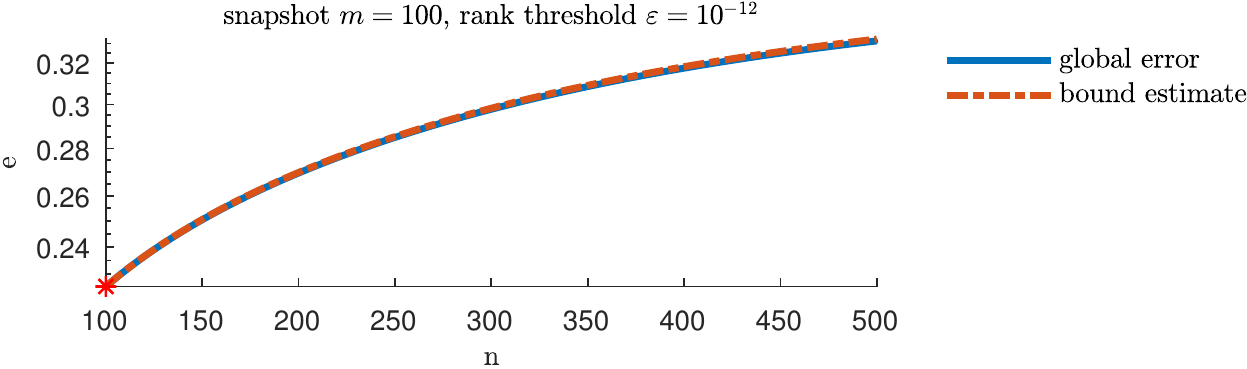}
\caption{Test~1a. Local truncation error for DMD with $m=100$, $200$ and $300$ snapshots (top); and global error $\bold e$ (error of the solution $u$)  for DMD with $m=100$ snapshots (bottom). The global error is negligible for $m = 200$ and $300$ (not shown). The rank threshold is set to $\varepsilon = 10^{-12}$.}
\label{fig:f4}
\end{figure}

\subsubsection{Periodic boundary fluctuations (Test 1b)}
Consider~\eqref{4-10} with $T = \pi/2$, $u_0 = 1$, and $u_\text{L} = 1.01+0.01\sin(-\pi/2+10t)$.
\cref{fig:f6} demonstrates that $m=200$ snapshots is sufficient for DMD to match the reference solution. The corresponding local and global truncation errors are plotted in \cref{fig:f7}. Since the solution $u(x,t)$ to~\eqref{4-10} with the parameter values used in Test 2 has a period of $\pi/5$, $m=100$ snapshots are not enough to cover the whole period. Consequently, DMD fails to capture the system dynamic and to predict the future states accurately. However, once the full period of the solution is covered by snapshots data, i.e., when $m=200$ or $300$ snapshots are used, DMD is accurate even for long-time prediction.  The error bound in~\cref{thm:t3-7} does a good job bounding the computed error. 

\begin{figure}[tbhp]
 \centering
\includegraphics{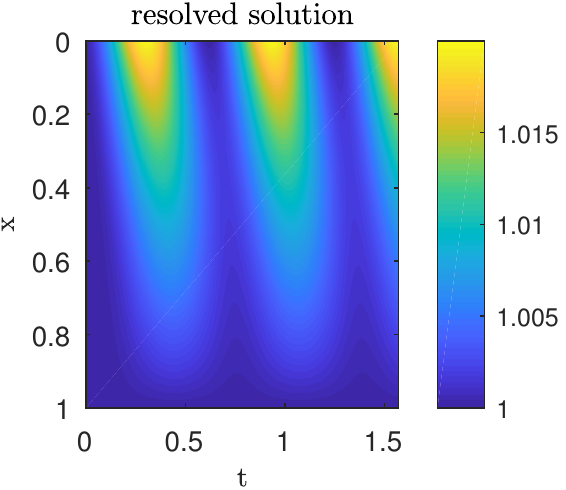}
\includegraphics{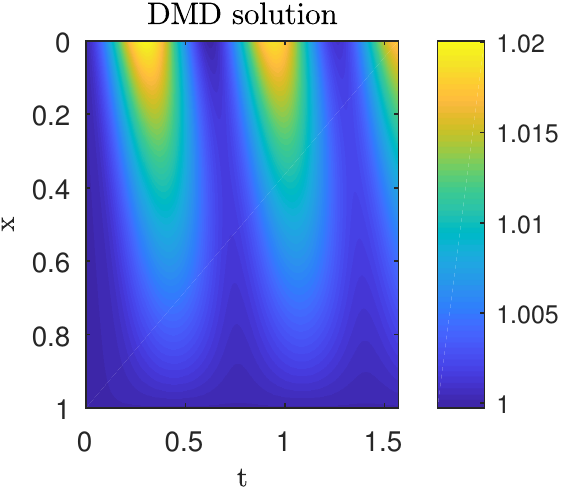}
\caption{Test 1b. Reference solution (left) and its DMD approximation with $m=200$ snapshots (right).}
\label{fig:f6}
\end{figure}

Although not shown here, the reliance on a more restricted rank truncation, i.e., setting the rank threshold to $\varepsilon = 10^{-12}$, improves DMD's accuracy by at least an order of magnitude for the parameter values considered.

\begin{figure}[tbhp]
 \centering
\includegraphics{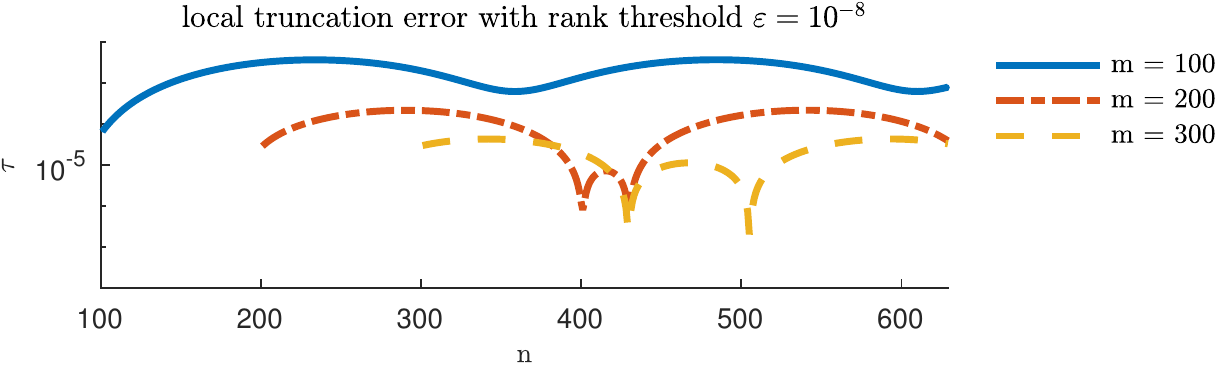}
\includegraphics{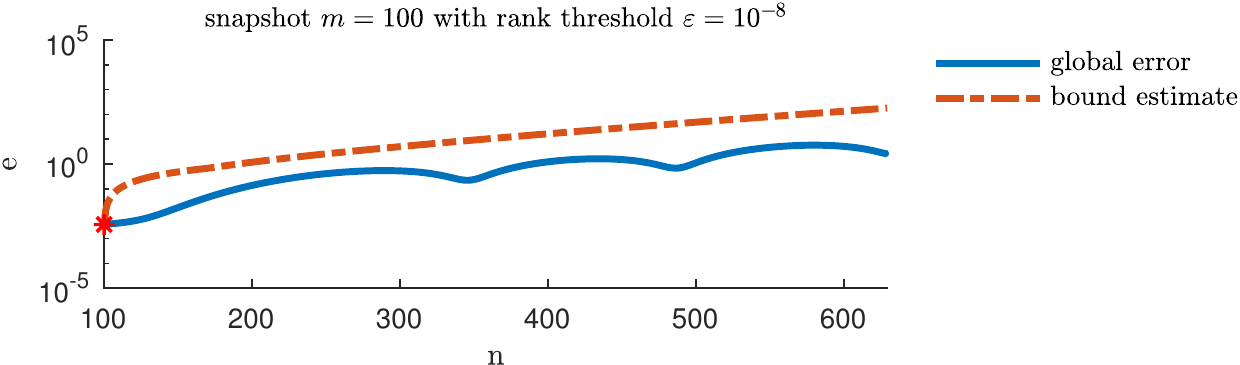}
\includegraphics{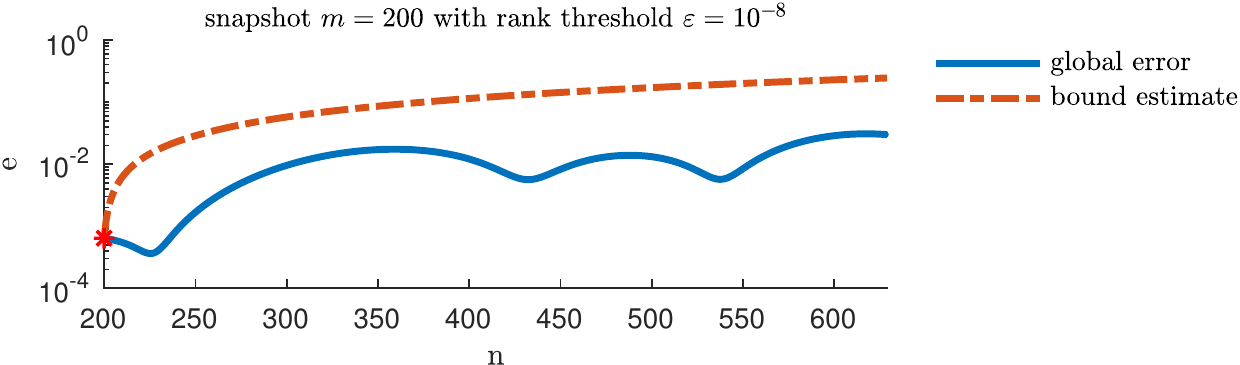}
\caption{Test 1b. Local truncation error $\tau$ for DMD with $m=100$, $200$ and $300$ snapshots (top); and global error $\bold e$ (error of the solution $u$) for $m=100$ (middle) and $m =200$ snapshots (bottom). The global error is negligible for $m = 300$ (not shown). The rank threshold is set to $\varepsilon = 10^{-8}$.}
\label{fig:f7}
\end{figure}

%
%
%

%
%
%
%

\subsection{Reaction-diffusion equation}
\label{sec:t2}
Consider a reaction-diffusion equation
\begin{subequations}\label{eq:rde}
\begin{equation}
\frac{\partial u}{\partial t} = \theta \frac{\partial^2 u}{\partial x^2} - \mu (u-u^3),  \qquad x\in[0,1], \qquad t\in [0,2]
\end{equation}
with constant coefficients $\theta, \mu \in \mathbb R^+$. It is subject to initial and boundary conditions
\begin{equation}
u(x,0)=0.5+0.5\sin(\pi x), \qquad u(0,t) = 0, \qquad u(1,t) = 0.
\end{equation}
\end{subequations}

\subsubsection{Diffusion-dominated regime (Test 2a)}
To achieve this regime ($\theta \gg \mu$), we set $\theta =0.1$ and $\mu = 0.01$. \cref{fig:f11} exhibits the fully resolved solution with its approximations provided by DMD with different observables, $g_1(u)= u$ and $\mathbf g_2(u) = (u, u^3)$, and by POD-DEIM. In~\cref{fig:f11}, the choice of observables does not appreciably affect DMD's performance due to the dominating linear diffusion, though one can still observe higher order accuracy of $\bold g_2$ than $g_1$ in the logarithm solution error plot~\cref{fig:f12}.

\begin{figure}[tbhp]
 \centering
\includegraphics{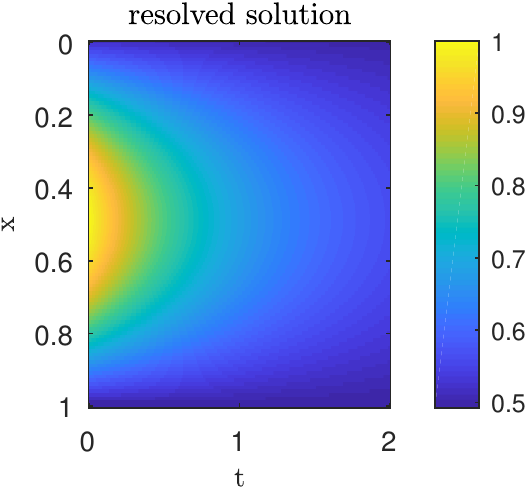}
\includegraphics{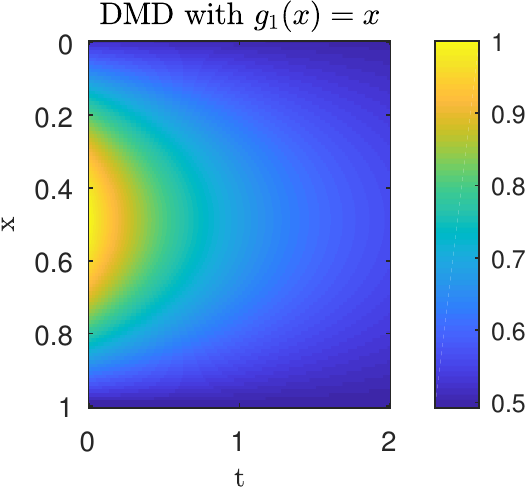}
\includegraphics{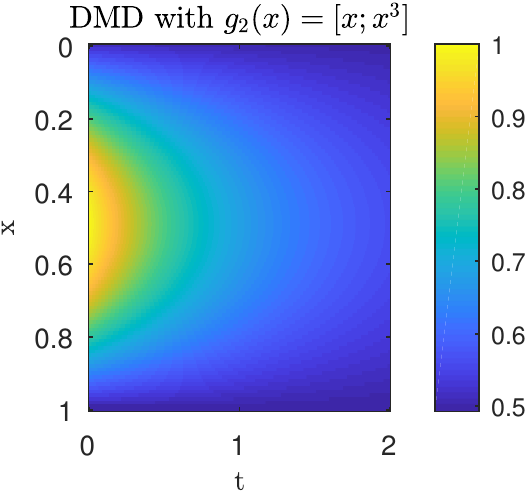}
\includegraphics{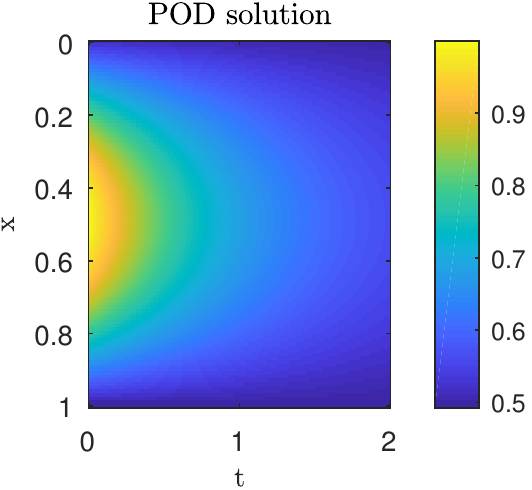}
\caption{Test 2a. Fully resolved solution $u(x,t)$ of the reaction-diffusion problem~\eqref{eq:rde} in the diffusion-dominated regime, and its approximations obtained from $m=200$ snapshots with DMD (with two sets of observables $\mathbf g$) and POD-DEIM.}
\label{fig:f11}
\end{figure}

The corresponding prediction errors are also reported in~\cref{fig:f12}. With the same rank truncation criteria, POD is more accurate than DMD, especially in the absence of ``right'' observables. However, DMD is much faster than POD. We report the computational costs comparison in~\cref{sec:compare}.

\begin{figure}[!htp]
\centering
\includegraphics{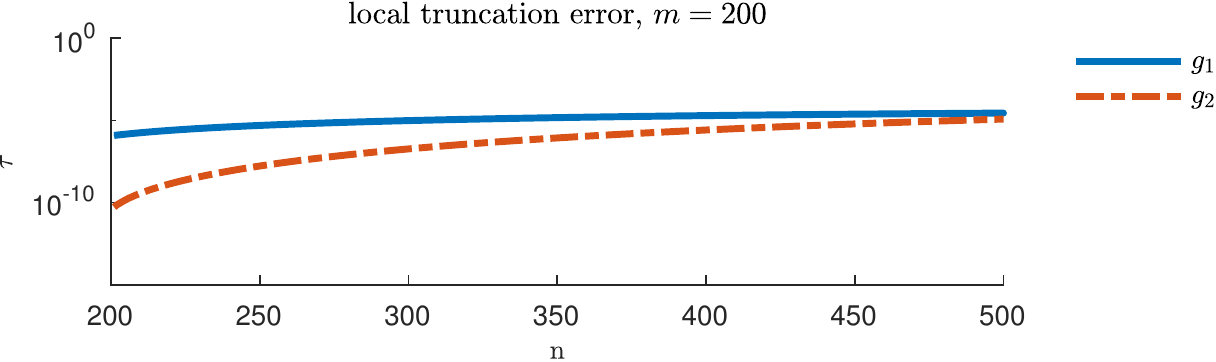}
\includegraphics{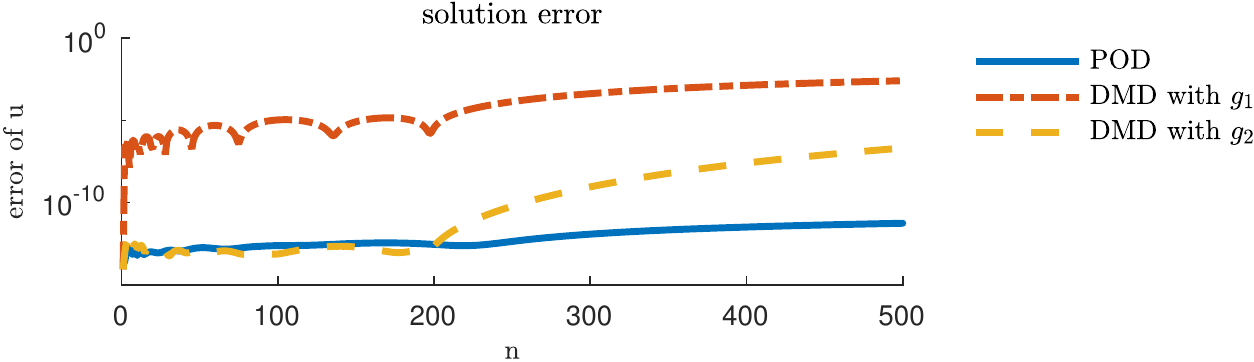}
\includegraphics{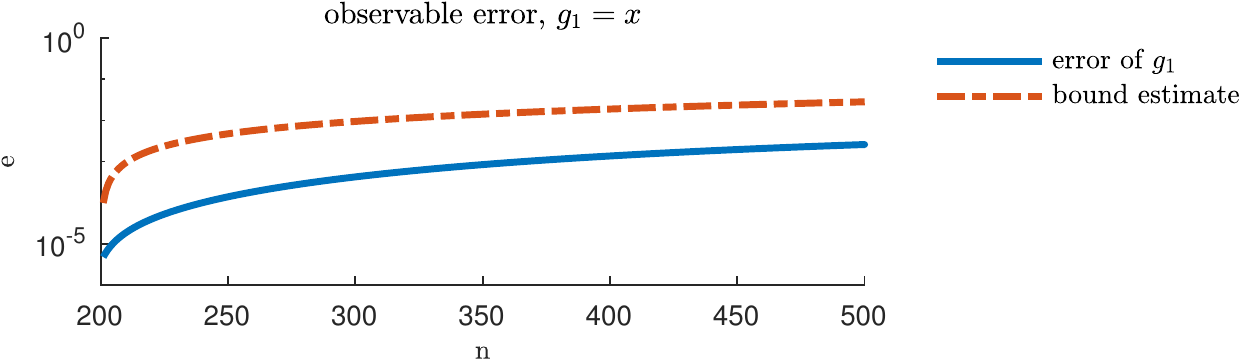}
\includegraphics{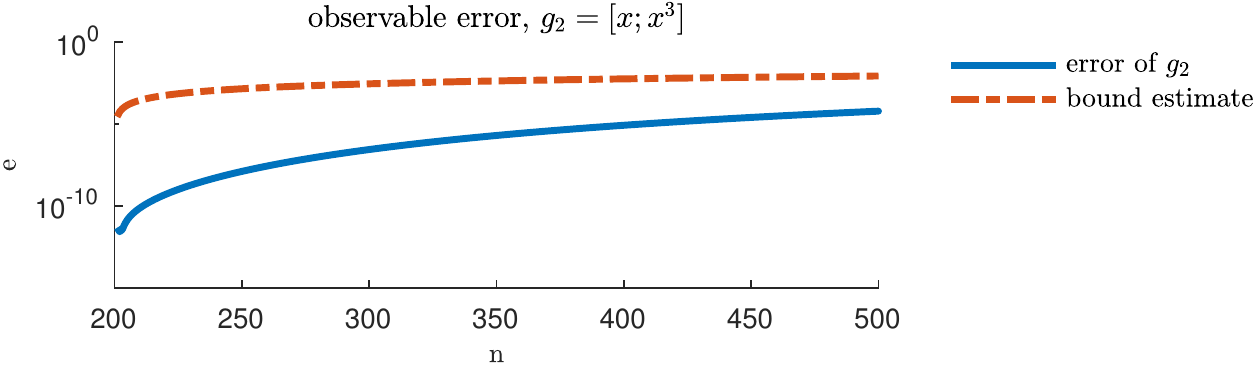}
\caption{Test 2a. Local truncation error; Comparison of POD and DMD errors of the solution; Global error (errors of the observables) for DMD prediction with observables $g_1$ and $\bold g_2$ using $m=200$ snapshots.}
\label{fig:f12}
\end{figure}

\subsubsection{Reaction-dominated regime (Test 2b)}

To explore this regime ($\mu\gg \theta$), we set $\theta = 0.1$ and $\mu = 1$. Now the choice of observables has significant (visual) impact on the predictive accuracy (\cref{fig:f13}). The Koopman operator theory helps explain this observation. Since the nonlinear source term dominates the dynamics, only the consistent observables can capture the eigenvalues and eigenfunctions of the Koopman operator.

\begin{figure}[tbhp]
 \centering
\includegraphics{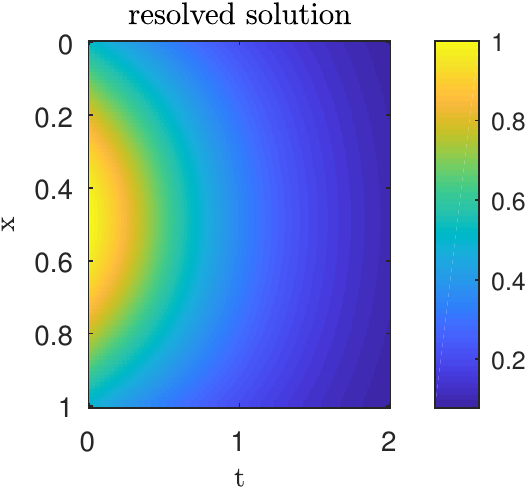}
\includegraphics{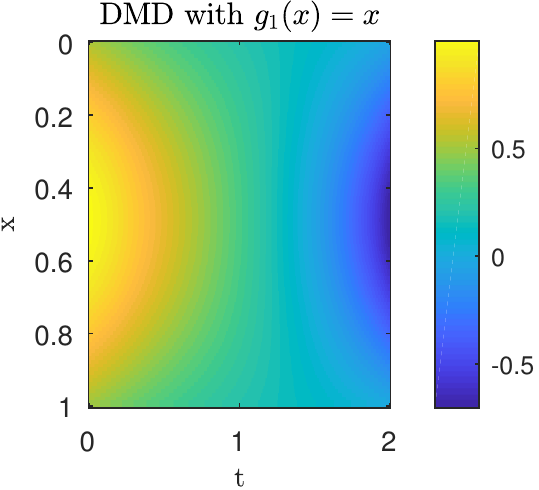}

\includegraphics{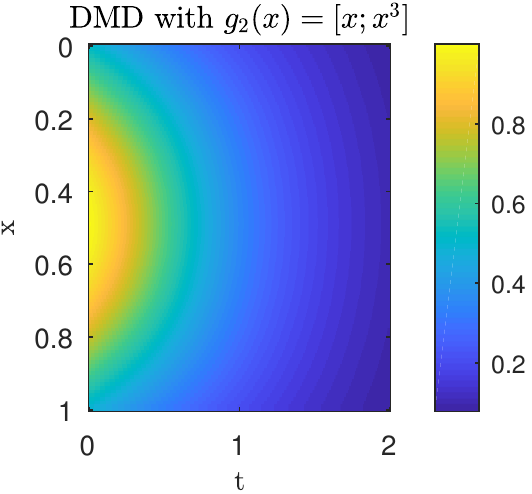}
\includegraphics{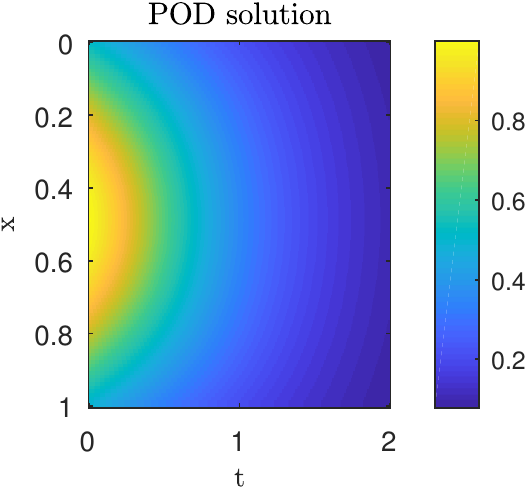}
\caption{Test 2b. Fully resolved solution $u(x,t)$ of the reaction-diffusion problem~\eqref{eq:rde} in the reaction-dominated regime, and its approximations obtained from $m=200$ snapshots with DMD (with two sets of observables $\mathbf g$) and POD-DEIM.}
\label{fig:f13}
\end{figure}

Errors of DMD prediction relying on the observables $g_1(u) = u$ and $\bold g_2(u) = (u, u^3)$ are shown in~\cref{fig:f14}. Our error estimation~\cref{thm:t3-7} indicates the failure of the DMD prediction based on the observable $g_1(u) = u$ and provides robust error bound for the DMD prediction based on  the observable $\bold g_2(u) = (u, u^3)$. For the same rank truncation criteria, the errors of POD and DMD with using $\bold g_2(u)$ are comparable, while that of DMD with $g_1(u)$ is orders of magnitude higher. 

\begin{figure}[tbhp]
\centering
\includegraphics{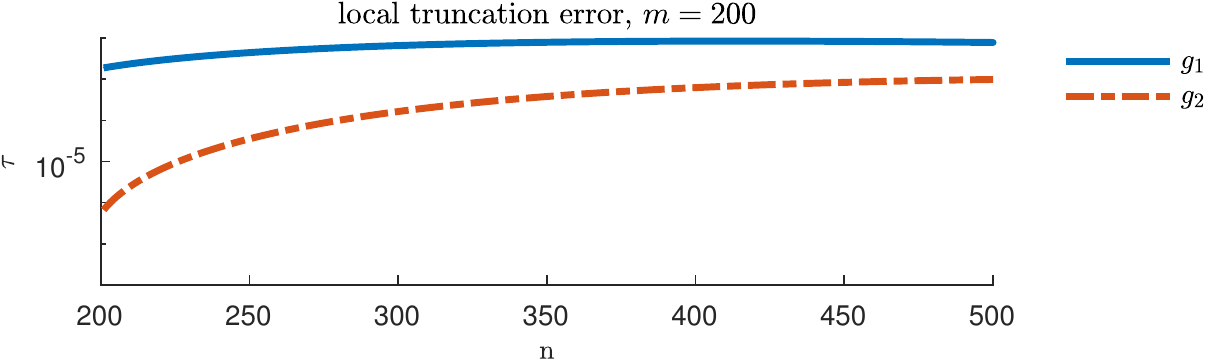}

\includegraphics{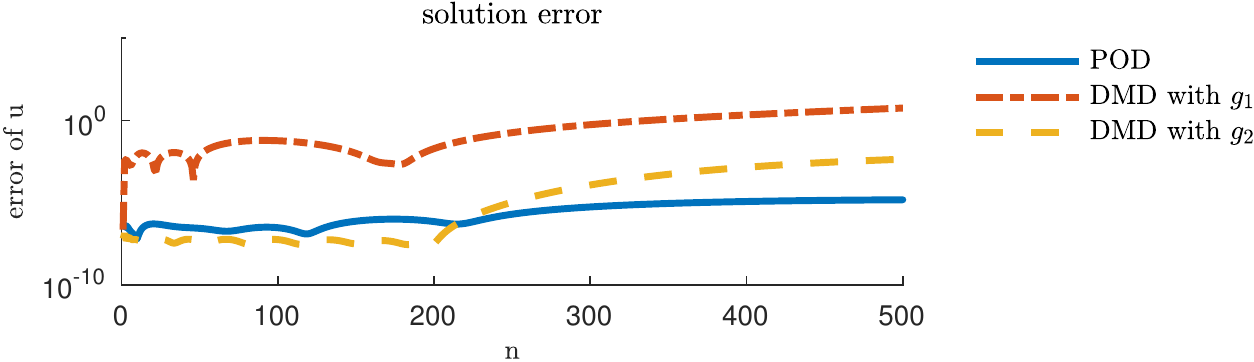}

\includegraphics{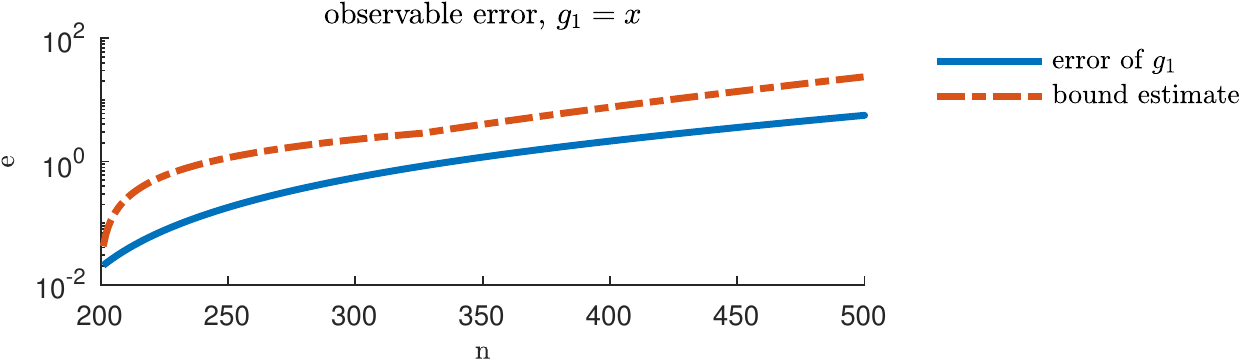}

\includegraphics{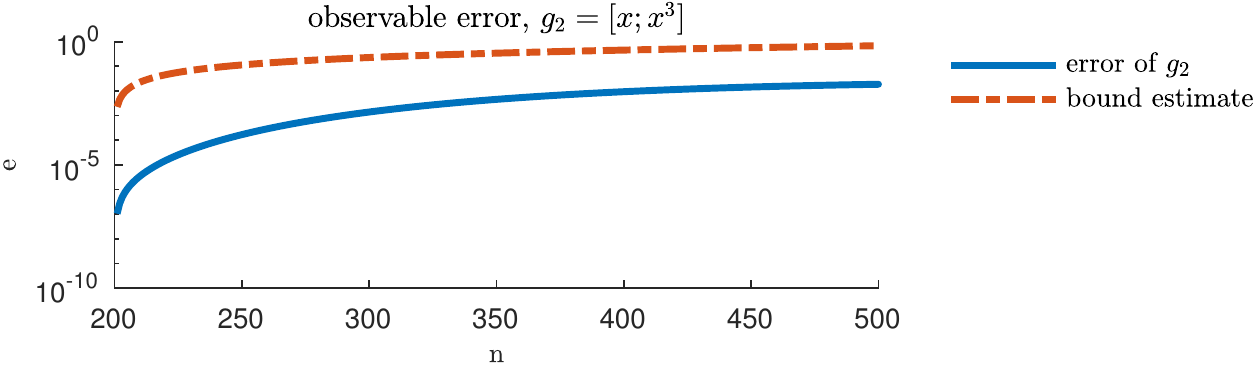}
\caption{Test 2b. Local truncation error; Comparison of POD and DMD errors of the solution; Global error (errors of the observables) for DMD prediction with observables $g_1$ and $\bold g_2$ using $m=200$ snapshots.}
\label{fig:f14}
\end{figure}

\subsubsection{Comparison of POD and DMD}
\label{sec:compare}
Comparison of the computation time and accuracy of DMD and POD-DEIM is presented in~\cref{table:t1} for Test~2b. The computational time comparison is made for the same rank truncation criteria. Note that the rank of the reduced-order model is different for DMD and POD because of the different dimension of the input data matrix. The ROM derived by DMD is in observable space and the ROM derived by POD is in state space. 

\begin{figure}[tbhp]
\centering
\includegraphics[width=\textwidth,clip=true, trim = 0mm 120mm 0mm 40mm]{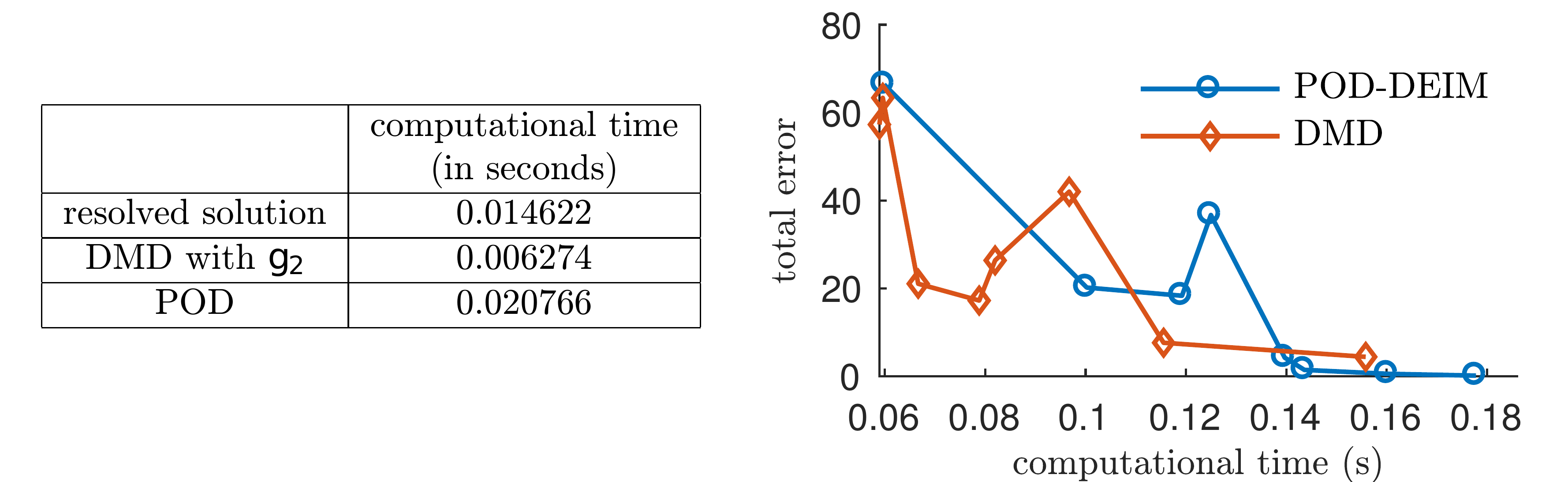}
\caption{Test 2b. Computational times of the fully resolved solution, POD-DEIM, and DMD with the observables $\mathbf g_2(u)$ (left table); Comparison of POD and DMD in terms of computational time and accuracy (right figure). \label{table:t1}}
\end{figure}

\cref{table:t1} demonstrates that DMD prediction is computationally efficient due to its iteration-free feature. POD, on the other hand, is computationally more expensive than the fully resolved solver because the computational cost saved by ROM in the prediction process does not compensate for the cost of establishing the ROM by SVD and DEIM. This would not be the case for higher dimensional problems and longer prediction times. However, being non-iterative, DMD would outperform POD on such problems as well.


Both the accuracy and computational time depend on the rank of the ROM. The table in \cref{table:t1} reveals that POD has advantage in accuracy and DMD has advantage in efficiency. Thus, if one wants a fast prediction with slightly lower accuracy, then DMD is a better choice and vice versa.

\subsection{Nonlinear reaction-diffusion equation (Test 3)}
\label{sec:t3}
Consider a reaction-diffusion equation with the state-dependent diffusion coefficient,
\begin{subequations}\label{prob3}
\begin{align}\label{eq:test3}
\frac{\partial u}{\partial t} = \frac{\partial }{\partial x}\left(u \frac{\partial u}{\partial x} \right) - (u-u^3), \qquad 0 < x < 1, \quad t > 0.
\end{align}
It is subject to the initial and boundary conditions
\begin{align}
u(x,0)=0.5+0.5\sin(\pi x), \quad u(0,t) = 0, \quad u(1,t) = 0.
\end{align}
\end{subequations}
As discussed earlier, the Koopman operator theory suggests that only the physical-informed observables can capture the dynamical systems.  To identify the relevant observables, we use the Kirchhoff transformation to recast~\eqref{eq:test3} as
\begin{align}
\frac{\partial u}{\partial t} = \frac{\partial^2 \phi }{ \partial x^2} - (u-u^3), \qquad \phi = 
u^2/2.
\end{align}
This form suggests a set of observables $\mathbf g_2 = (u; u^2; u^3)$.

\begin{figure}[tbhp]
 \centering
\includegraphics{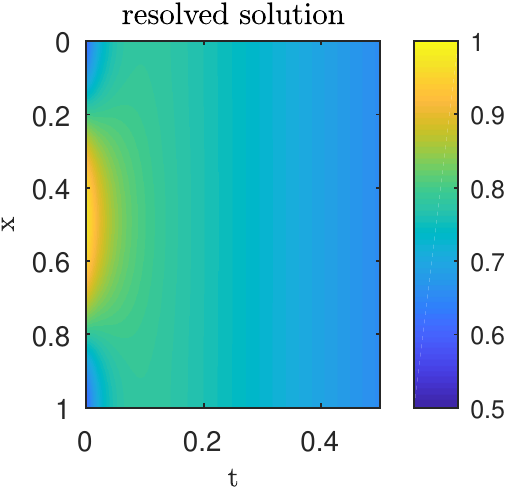}
\includegraphics{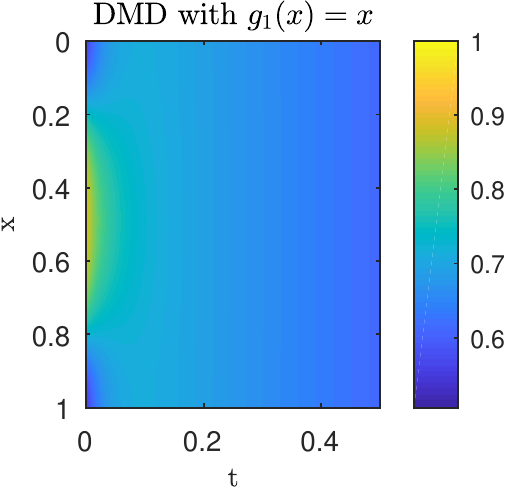}

\includegraphics{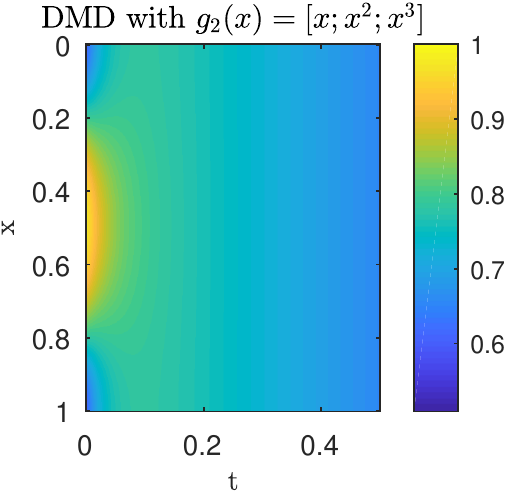}
\includegraphics{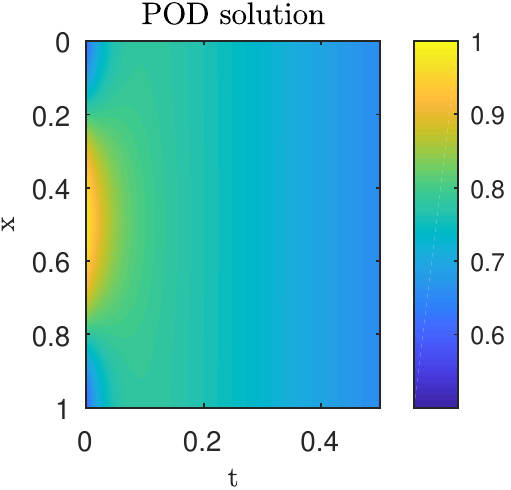}
\caption{Test 3. Fully resolved solution $u(x,t)$ of the nonlinear reaction-diffusion problem~\eqref{prob3} and its approximations obtained from $m=200$ snapshots with DMD (with two sets of observables $\mathbf g$) and POD-DEIM.}
\label{fig:f16}
\end{figure}

\Cref{fig:f16,fig:f17} provide the visual and quantitative comparison between the fully resolved solution $u(x,t)$ and its POD and DMD approximations. The performance of these approximators on this highly nonlinear problem is qualitatively similar to its weakly nonlinear counterpart analyzed in \cref{sec:t2}. For the inadequate choice of observables, $g_1 = u$, our error bound diverges from the true error because of the fast decay of both the reference and wrong solutions. Nevertheless, the error bound still serves as a good discriminator between the accurate or inaccurate predictions. For the proper choice of observables, $\mathbf g_2 = (u; u^2; u^3)$, our error bound remains accurate.

\begin{figure}[tbhp]
\centering
\includegraphics{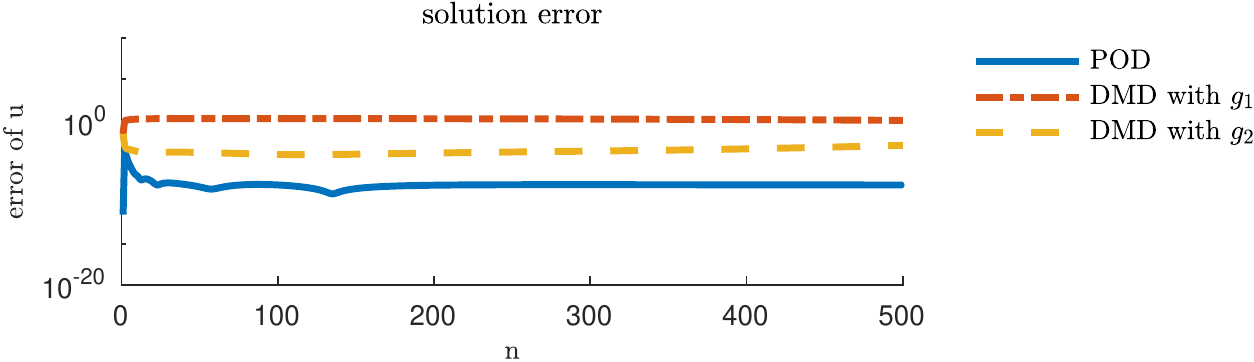}

\includegraphics{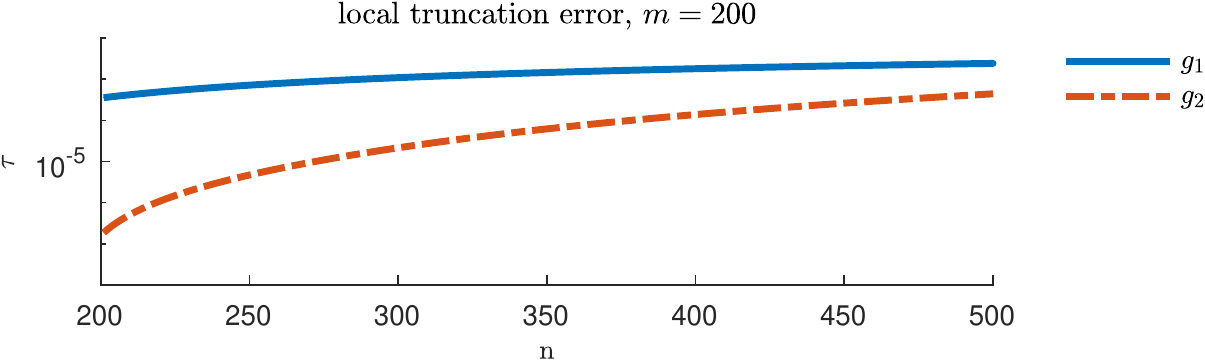}

\includegraphics{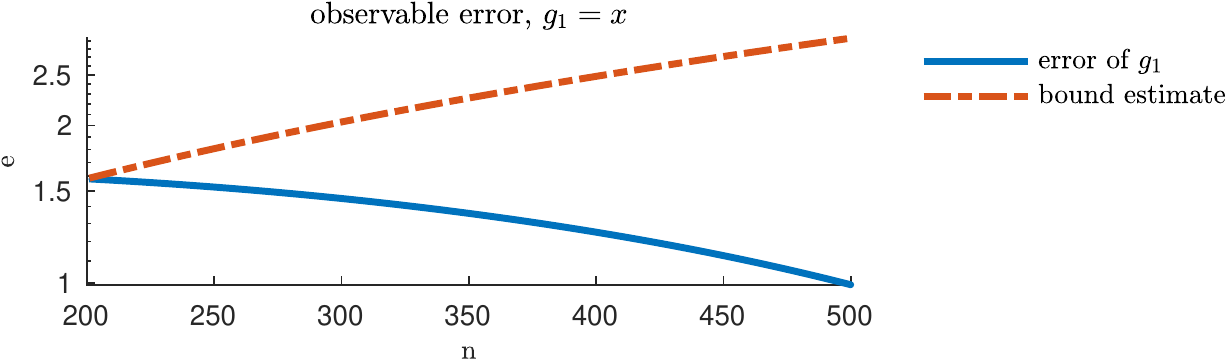}

\includegraphics{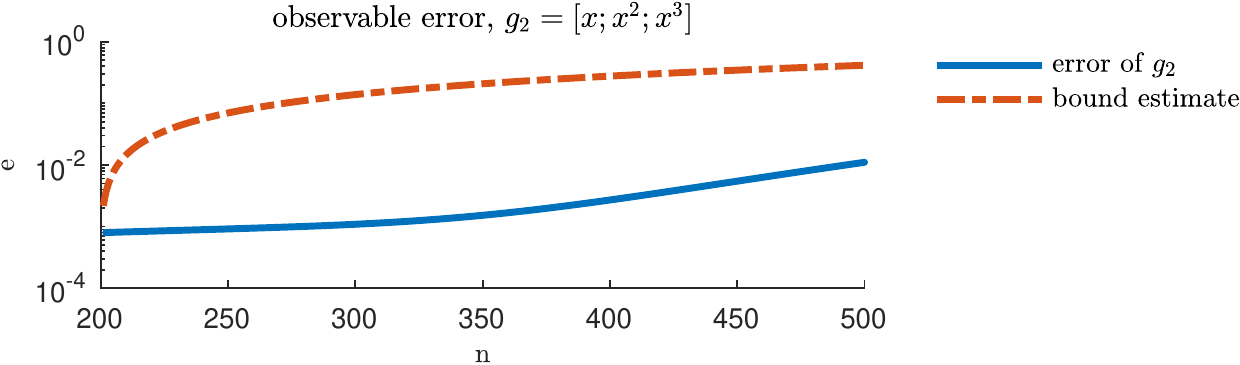}
\caption{Test 3. Comparison of POD and DMD errors; Local truncation error and global error for DMD prediction with $\bold g_1$ and $\bold g_2$ using $m=200$ snapshots.}
\label{fig:f17}
\end{figure}

\subsection{Nonlinear Schr\"odinger equation (Test 4)}
\label{sec:t4}
Finally, we consider the nonlinear Schr\"odinger  equation, 
\begin{equation}
\left\{
\begin{aligned}
&i\frac{\partial q}{\partial t}+\frac{1}{2}\frac{\partial^2 q}{\partial \xi^2}+|q|^2q = 0, \\
&q(x,0)=2 \text{sech}(x).
\end{aligned}
\right.
\end{equation}
It belongs to the the general class of nonlinear parabolic PDEs~\cref{eq:2-1} and satisfies all of the assumptions underlying our error estimator. The reference solution is obtained by using Fast Fourier Transform in space and Runge-Kutta in time evolution. 

\begin{figure}[tbhp]
 \centering
\includegraphics{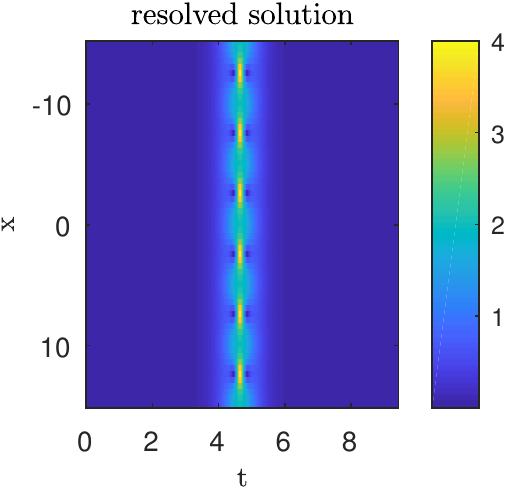}
\includegraphics{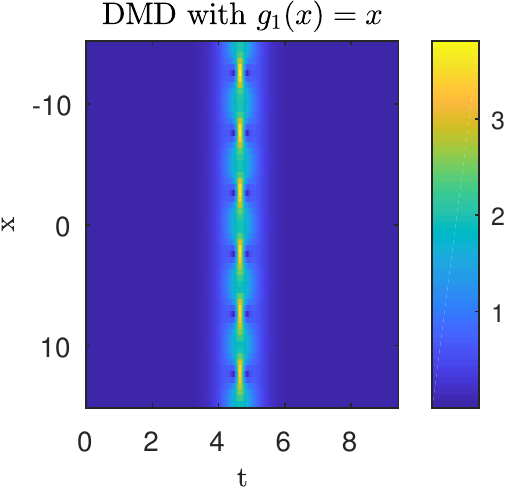}

\includegraphics{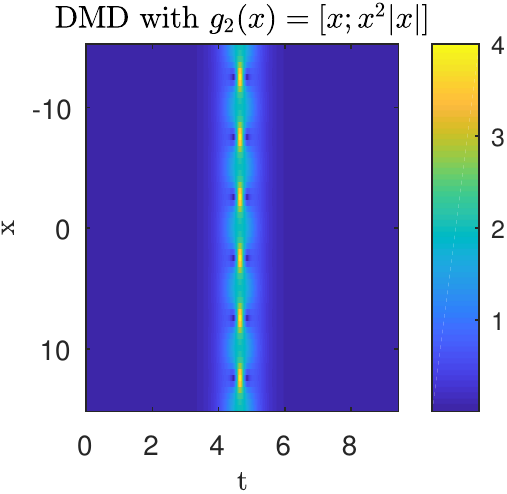}
\includegraphics{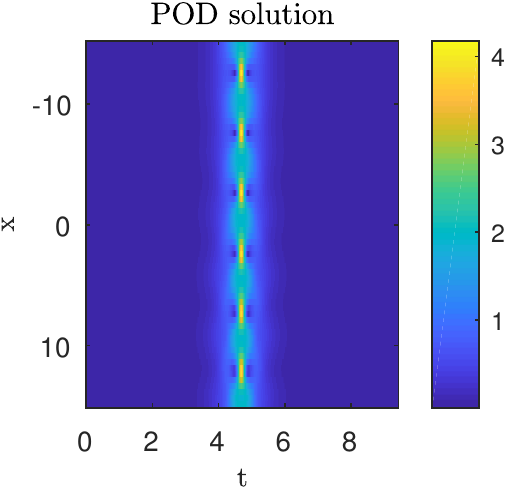}
\caption{Test 4. Resolved solution, DMD solutions and POD solution using $m=20$ snapshots; Comparison of POD and DMD errors.}
\label{fig:f18}
\end{figure}

We reproduce the results reported in~\cite{kutzbook} and use them to verify our error bound in~\cref{fig:f18,fig:f19}. In this case, DMD with the right observable has better performance, in terms of both accuracy and efficiency, than POD. The advantage of taking physical information into account is tremendous.

\begin{figure}[tbhp]
\centering
\includegraphics{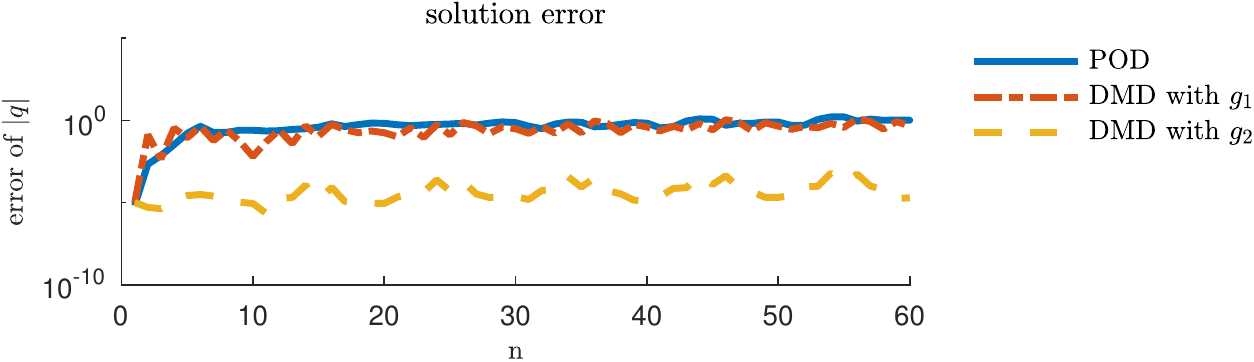}

\includegraphics{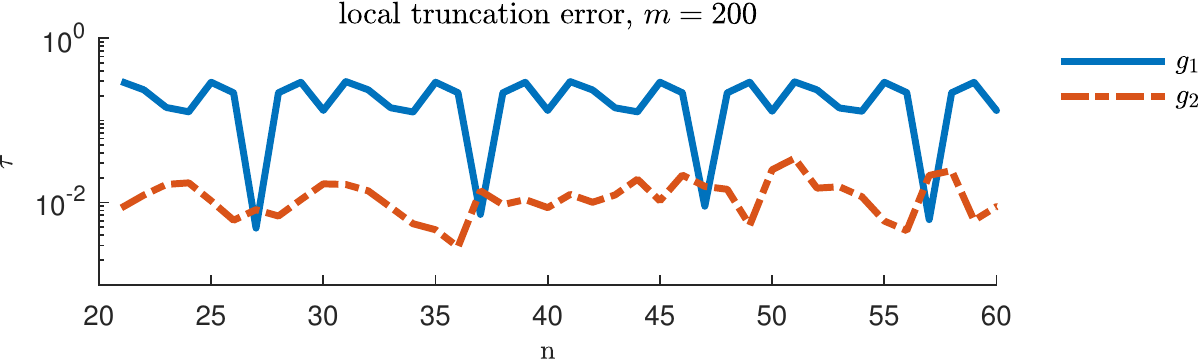}

\includegraphics{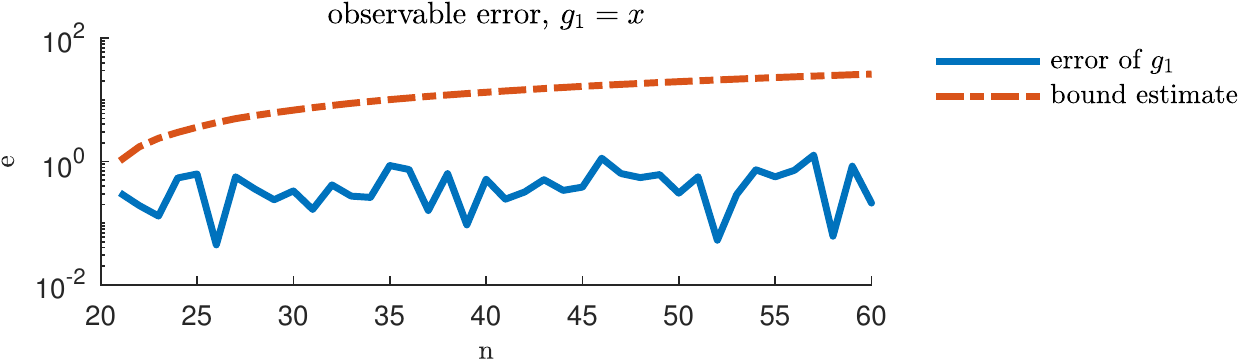}

\includegraphics{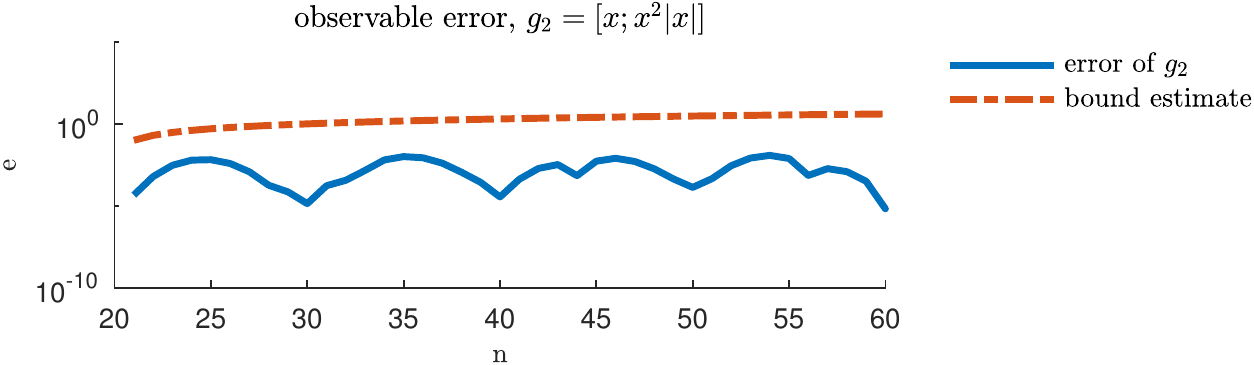}
\caption{Test 4. Comparison of POD and DMD errors; Local truncation error and global error for DMD prediction with $\bold g_1$ and $\bold g_2$ using $m=200$ snapshots.}
\label{fig:f19}
\end{figure}

\section{Conclusion and Future outlook}
\label{sec:con}

We derived error bounds of DMD predictions for linear and nonlinear parabolic PDEs and verified their accuracy on four computational examples of increasing degree of complexity. Our analysis leads to the following major conclusions.
\begin{enumerate}
\item When combined with an adequate choice observables, the Koopman operator maps the nonlinear underlying dynamics with the linear observable space, where DMD algorithm can be implemented with good accuracy and efficiency. 

\item In the extrapolation (predictive) mode, DMD outperforms other ROM-based method (e.g., POD) in terms of computational efficiency, because it requires no iteration. At the same time, POD has higher predictive accuracy than DMD.


\item Our error estimator is consistent with previous theoretic understanding of DMD algorithm and the Koopman operator theory. More importantly, it provides a quantitative measure of the accuracy of DMD predictions. 
\end{enumerate}
In the follow-up studies we will used  our error estimators of DMD predictions to address several challenges in scientific computing:
\begin{enumerate}
\item For PDEs with random coefficients, e.g., for PDE-based models of flow and transport in (randomly) heterogeneous porous media, DMD predictions with quantitative error bounds might provide a means for accelerating computational expensive Monte Carlo and multiscale simulations.
\item Our error estimators can be used to guide the design of hybrid algorithms that combine DMD predictions with fully resolved solutions of multi-dimensional complex problems. 

\item It might be possible to generalize our results to a broader context of advection-diffusion equations. Multiresolution DMD (mrDMD), instead of DMD, can be used to overcome the translational invariant issues in advection. 
\end{enumerate}

\bibliographystyle{plain}

\end{document}